\documentclass{article}

\usepackage{latexsym}
\usepackage{amssymb,amsmath, amsthm, amsfonts}
\usepackage{graphicx}
\usepackage{color}

\makeatletter
\providecommand*{\Dashv}{%
  \mathrel{%
    \mathpalette\@Dashv\vDash
  }%
}
\newcommand*{\@Dashv}[2]{%
  \reflectbox{$\m@th#1#2$}%
}
\makeatother

\newtheorem{theorem}{Theorem}
\newtheorem{proposition}{Proposition}
\newtheorem{lemma}{Lemma}
\newtheorem{corollary}{Corollary}

\theoremstyle{definition}
\newtheorem{definition}{Definition}
\newtheorem{remark}{Remark}
\newtheorem{example}{Example}

\newcommand{\algLfin}{\textbf{\L}\mathbf{V}}
\newcommand{\domLfin}{{\L}V}

\newcommand{\ntwrt}[1]{non-trivial with respect to \ensuremath{#1}}

\newcommand{\nvert}{\mathop{\!\not\vert}}

\begin{document}

\title{
{Maximality
in finite-valued {\L}ukasiewicz logics defined by order filters}
}
\author{Marcelo E. Coniglio$^1$, Francesc Esteva$^2$, Joan Gispert$^3$ and Lluis Godo$^2$ \\ \\
\small $^1$ Dept. of Philosophy - IFCH and  \\ \small  Centre for Logic, Epistemology and the History of Science, \\ \small University of Campinas, Brazil \\
\small {\tt coniglio@cle.unicamp.br} \\
\small $^2$ Artificial Intelligence Research Institute (IIIA) - CSIC, Barcelona, Spain \\
\small {\tt \{esteva,godo\}@iiia.csic.es} \\
\small $^3$ Dept. of Mathematics and Computer Science, University of Barcelona, Spain \\
\small {\tt jgispertb@ub.edu}
}
\date{}
\maketitle

\begin{abstract}
 In this paper we consider the logics $\mathsf{L}_n^i$ obtained from the $(n+1)$-valued  {\L}ukasiewicz logics {\L}$_{n+1}$ by taking the order filter generated by $i/n$ as the set of designated elements. In particular, the conditions of maximality and strong maximality among them are analyzed. We present a very general theorem which provides sufficient conditions for maximality between logics. As a consequence of this theorem it is shown that $\mathsf{L}_n^i$ is maximal w.r.t.\ {\sf CPL} whenever $n$ is prime. Concerning strong maximality between the logics $\mathsf{L}_n^i$ (that is, maximality w.r.t. rules instead of axioms), we provide algebraic arguments in order to show that the logics $\mathsf{L}_n^i$ are not strongly maximal w.r.t.\ {\sf CPL}, even for $n$ prime. Indeed, in such case, we show there is just one extension between $\mathsf{L}_n^i$ and {\sf CPL} obtained by adding to $\mathsf{L}_n^i$  a kind of graded explosion rule.  Finally, using these results, we show that the logics  $\mathsf{L}_n^i$ with $n$ prime and $i/n < 1/2$ are ideal paraconsistent logics.
\end{abstract}

\section{Introduction}

In this paper we study the notion of  maximality and strong maximality among finite-valued propositional logics. Recall the usual notion of maximality found in the literature:  a propositional logic $L_1$,  that is a sublogic of another logic $L_2$  (in the sense of inclusionship of their consequence relations over the same signature), is called maximal
with respect to $L_2$  if, roughly speaking, $L_1$ extended with any theorem of $L_2$ which is not a theorem of $L_1$, coincides with $L_2$. Similarly, recall the stronger notion of strong maximality following \cite{ArieliAZ10,AvronAZ10,RibCon12}:  $L_1$ is called strongly maximal  with respect to $L_2$  if, roughly speaking again, $L_1$ extended with  a rule of inference valid in $L_2$ but not a valid in $L_1$, coincides with $L_2$.

The problem of finding and characterizing maximal sublogics (in both senses) of a given logic has already been addressed in the literature, specially in the context of paraconsistent logics, where being maximal with respect to classical logic is felt as a desirable or ideal feature, c.f. \cite{ArieliAZ11a,CarCon16}. Indeed,  being maximal means that, while still allowing non-trivial inconsistent theories, it retains as much as possible of classical logic.

In the present paper we approach the general problem of characterizing maximality (not necessarily for paraconsistent logics) in two different scenarios. The first one considers a very general class of finite-valued logics, those defined by almost arbitrary finite logical matrices. In such a context, we provide a sufficient condition for a logic to be maximal w.r.t. another one with less truth-values under very general conditions. This result, inspired on the notion of recovery operators from paraconsistent logics, turns out to be very powerful and encompasses many maximality results scattered in the literature.

The second scenario considers a particular class of finite-valued logics, the class of $n$-valued {\L}ukasiewicz logics \L$_{n}$ and their related logics defined by order filters. We show that these logics, for $n$ being prime, are maximal but not strongly maximal with respect to classical logic. Actually, we show that each of these logics can always be uniquely extended with a sort of explosion inference rule such that the obtained logic is the unique one below classical logic, and hence strongly maximal.

The paper is structured as follows. After this introduction, we provide in Section \ref{recovery} a very general condition for a finite matrix logic to be maximal w.r.t. another one with less truth-values, and we analyze in particular the case of 3-valued logics. In the rest of the paper we focus our attention on the class of finite-valued {\L}ukasiewicz logics $\mathsf{L}_n^i$ defined by order filters. In Section \ref{Sect-max}  we identify which of these logics are maximal with respect to classical logic, while in Section \ref{sectLqi} we study their status regarding the property of strong maximality.  It is in Section \ref{Joan} where we fully characterize, by algebraic techniques, conditions of strong maximality.  {Finally, in Section~\ref{sectIdeal} the question of ideal paraconsistent logics (as introduced in~\cite{ArieliAZ11a}) will be analized in the present framework. Specifically, it will be shown that the logics  $\mathsf{L}_n^i$ with $n$ prime and $i/n < 1/2$ are ideal paraconsistent logics. In addition, the case $\mathsf{L}_3^1$ will be discussed with more detail, and it will be argued that this logic constitutes the 4-valued version of the well-known 3-valued paraconsistent logic   $\mathsf{J}_3$ (see~\cite{dot:dac:70}).}
We finish in Section \ref{concl} with some conclusions and prospects of future research.

\section{Maximality and recovery operators} \label{recovery}

Let us recall the usual notion of maximality of a  (standard) logic with respect to another:

\begin{definition}
Let $L_1$  and $L_2$ two standard propositional logics defined over the same signature $\Theta$ such that $L_1$  is a proper sublogic of  $L_2$, i.e.\  such that ${\vdash_{L_1}} \subsetneq {\vdash_{L_2}}$, where $\vdash_{L_i}$ denotes the consequence relation of  $L_i$ (for $i=1,2$). Then, $L_1$ is said to be {\em maximal} w.r.t.\ $L_2$ if, for every formula $\varphi$ over $\Theta$, if $\vdash_{L_2} \varphi$ but $\nvdash_{L_1} \varphi$, then the logic $L_1^+$  obtained from  $L_1$   by adding $\varphi$ as a theorem, coincides with $L_2$.
\end{definition}

By $L_1^+$  above we mean the logic whose consequence relation is obtained from the one of  $L_1$ as follows: for every set of formulas $\Gamma \cup \{\psi\}$ over $\Theta$,
 $$\Gamma \vdash_{L_1^+} \psi \ \ \mbox{ if }  \Gamma, \{\sigma(\varphi) \ : \ \sigma \ \mbox{is a substitution over $\Theta$}\} \vdash_{L_1} \psi.$$

\begin{remark} \label{vacuous}
It should be noticed that, according to the above definition, if $L_1$ is a proper sublogic of $L_2$ such that they validate the same formulas (that is: $\vdash_{L_1} \varphi$  \ iff \  $\vdash_{L_2} \varphi$, for every  formula $\varphi$) then $L_1$ is maximal w.r.t. $L_2$.
\end{remark}

In this section, for the class of  propositional logics induced by finite logical matrices, we will provide a very general sufficient condition for a logic to be maximal w.r.t.\ another one (see Theorem \ref{maxthm} below), its proof being inspired in the role played by the so-called recovery operators in paraconsistent and adaptive logics. Recall from~\cite{CM} (see also~\cite{car:con:mar:07,CarCon16})  the definition of the class of paraconsistent logics called {\em Logics of Formal Inconsistency} ({\bf LFI}s):  a given logic, say $L$, is an {\bf LFI} if it is paraconsistent w.r.t. some negation, say $\neg$ (that is, there exist formulas $\alpha$ and $\beta$ such that $\beta$ does not follows from $\{\alpha,\neg\alpha\}$ in $L$). In addition, there is a  (primitive of definable) unary connective $\circ$ in $L$ (called a {\em consistency operator}) such that every formula $\beta$ follows in $L$ from a set of the form$\{\alpha,\neg\alpha, \circ\alpha\}$.\footnote{This is a slightly simplified presentation of the original definition of {\bf LFI}s.}  If $L$ is an  {\bf LFI} which is sublogic of  classical propositional logic ({\sf CPL}), presented in the same signature of $L$,\footnote{In this case, the formulas $\circ \alpha$ take the value 1 for every evaluation in  {\sf CPL}.}  then the consistency operator $\circ$ allows to recover {\sf CPL} inside  $L$ by adding additional hypothesis concerning the consistency (or `classicality', or `well-behavior') of some formulas. Namely, for every (finite) set $\Gamma \cup \{\psi\}$ of formulas,
$$\Gamma \vdash_{\sf CPL} \psi \ \mbox{ iff } \ (\exists \Lambda)[\Gamma, \{\circ \alpha \ : \  \alpha \in \Lambda\} \vdash_L \psi],$$
where $\Lambda$ is a set of formulas. This is what is called  a {\em Derivability Adjustment Theorem} (DAT).
The idea of DATs was proposed by Battens in the context of {\em Adaptive logics}, but this technique (as well as the notion of consistency operator) was already used by da Costa for his well-known hierarchy of paraconsistent systems $C_n$ (see~\cite{dac:63}).

A more interesting DAT (as, for instance, the ones obtained by da Costa) requires that the consistency (or well-behavior) operator $\circ$ can just be applied to the propositional variables occurring in  $\Gamma \cup \{\psi\}$. This suggests that, given two standard propositional logics $L_1$  and $L_2$ defined over the same signature $\Theta$ such that ${\vdash_{L_1}} \subseteq {\vdash_{L_2}}$, a DAT between both logics can be defined in terms of a {\em recovery} operator $\circ$ (generalizing the idea of {\bf LFI}s):
 for every (finite) $\Gamma \cup \{\psi\}$,
$$\Gamma \vdash_{L_2}\psi \ \mbox{ iff } \ \Gamma, \{\circ p_1, \ldots,\circ p_m\} \vdash_{L_1} \psi,$$
where $\{p_1, \ldots, p_m\}$ is the set of propositional variables occurring in $\Gamma \cup \{\psi\}$.

The idea then is that if one of such recovery operators $\circ_\varphi$ can be defined as  a family of instances of a theorem $\varphi$ of $L_2$ which is not derivable in $L_1$, and if this process can be reproduced for any of such formulas $\varphi$, then it will follow that $L_2$ is maximal w.r.t.\ $L_1$. To be more general, a finite recovery set $\bigcirc(p)$ of formulas depending only on one variable $p$ will be considered instead of a single formula $\circ(p)$, following the original definition of {\bf LFI}s. Actually, in Theorem~\ref{maxthm} below some sufficient conditions are given in order to define such recovery sets, which will allow us to determine if one logic is maximal w.r.t.\ another.

In what follows, $\mathcal{L}(\Theta)$ will denote the term algebra generated by a propositional signature $\Theta$ from a fixed set $P=\{p_n  \ : \ n \geq 1\}$ of propositional variables. If {\bf A} is an algebra over $\Theta$ then the set of homomorphisms from $\mathcal{L}(\Theta)$ to {\bf A} will be denoted by $Hom(\mathcal{L}(\Theta),{\bf A})$.

Given an algebra {\bf A} over $\Theta$ and a non-empty subset $F \subseteq A$, the pair $\langle {\bf A}, F\rangle$ is called a {\em logical matrix} \cite{Woj88}. The logic $L$ defined by the matrix $\langle {\bf A}, F\rangle$ over $\mathcal{L}(\Theta)$ is given by the following consequence relation: for every set of formulas $\Gamma \cup\{\varphi\} \subseteq \mathcal{L}(\Theta)$,
$$ \Gamma \vdash_L \varphi  \mbox{ if,  for all } e\in Hom(\mathcal{L}(\Theta),{\bf A}), e(\psi) \in F \mbox{ for all } \psi \in \Gamma \mbox{ implies } e(\varphi) \in F . $$
From now on, with no danger of confusion,  given a logical matrix $\langle {\bf A}, F\rangle$  we will write $L = \langle {\bf A}, F\rangle$ to refer to the corresponding induced logic defined as above. We will also use the term {\em matrix logic} to refer a logic defined by a logical matrix.

\begin{lemma} \label{IncMat}
Let $L_1=\langle {\bf A}_1, F_1\rangle$  and $L_2=\langle {\bf A}_2, F_2\rangle$ be two matrix logics defined over a signature $\Theta$ such that
${\bf A}_2$ is a subalgebra of ${\bf A}_1$ and $F_2 = F_1 \cap A_2$. Then ${\vdash_{L_1}} \subseteq {\vdash_{L_2}}$, that is: for every $\Gamma \cup \{\psi\}$, if $\Gamma \vdash_{L_1} \psi$ then $\Gamma \vdash_{L_2} \psi$.
\end{lemma}
\begin{proof}
Assume that $\Gamma {\vdash_{L_1}} \psi$. Let $e \in Hom(\mathcal{L}(\Theta),{\bf A}_2)$ be an evaluation for $L_2$ such that $e[\Gamma] \subseteq F_2$. Let $\bar e:\mathcal{L}(\Theta)\to A_1$ such that $\bar e(\varphi)=e(\varphi)$ for every $\varphi \in \mathcal{L}(\Theta)$. Then $\bar e \in Hom(\mathcal{L}(\Theta),{\bf A}_1)$, so $\bar e$ is an evaluation for $L_1$ such that $\bar e[\Gamma] \subseteq F_1$. By hypothesis, $\bar e(\psi) \in F_1$ and so $e(\psi) \in F_1 \cap A_2=F_2$. This shows that   $\Gamma \vdash_{L_2} \psi$.
\end{proof}

After this previous lemma, we can state the main result on this section.

\begin{theorem} \label{maxthm}
Let $L_1=\langle {\bf A}_1, F_1\rangle$  and $L_2=\langle {\bf A}_2, F_2\rangle$ be two distinct finite matrix logics over a same signature $\Theta$ such that ${\bf A}_2$ is a subalgebra of ${\bf A}_1$ and $F_2 = F_1 \cap A_2$.
 Assume the following:
\begin{enumerate}
\item $A_1=\{0,1,a_1,\ldots,a_k,a_{k+1},\ldots,a_n\}$   and  $A_2=\{0,1,a_1,\ldots,a_k\}$ are finite such that $0 \not\in F_1$, $1 \in F_2$ and $\{0,1\}$ is a subalgebra of ${\bf A}_2$.
\item There are formulas $\top(p)$ and $\bot(p)$ in $\mathcal{L}(\Theta)$ depending at most on one variable $p$ such that $e(\top(p))=1$ and $e(\bot(p))=0$, for every evaluation $e$ for $L_1$.
\item For every $k+1 \leq i \leq n$ and $1 \leq j \leq n$ (with $i \neq j$) there exists a formula $\alpha^i_j(p)$ in $\mathcal{L}(\Theta)$ depending  at most on one variable $p$ such that, for every evaluation $e$, $e(\alpha^i_j(p))=a_j$ if $e(p)=a_i$. \end{enumerate}
Then,  $L_1$ is maximal w.r.t.\  $L_2$.
\end{theorem}
\begin{proof}
Let us begin by  observing that the family of evaluations for $L_1$ which take values in $A_2$ for every propositional variable can be identified with the family of evaluations for $L_2$.\footnote{This fact was used in the proof of Lemma~\ref{IncMat}.}

Notice that, by  Lemma~\ref{IncMat}, ${\vdash_{L_1}} \subseteq {\vdash_{L_2}}$. Suppose that there is some formula $\varphi(p_1,\ldots,p_m)$  such that $\vdash_{L_2} \varphi$ but $\nvdash_{L_1} \varphi$ (otherwise the proof is done, by Remark~\ref{vacuous}). Then, $e(\varphi) \in F_2$ for every evaluation $e \in Hom(\mathcal{L}(\Theta),{\bf A}_2)$, but there is an homomorphism $e_0 \in Hom(\mathcal{L}(\Theta),{\bf A}_1)$ such that $e_0(\varphi) \not\in F_1$. By the observation at the beginning of the proof (and by considering that $F_2 \subseteq F_1$), there exists a propositional variable $p_i$ (for $1 \leq i \leq m$) such that $e_0(p_i) \not\in A_2$. Consider now a substitution $\sigma_0$ such that

$$\sigma_0(p)= \left \{ \begin{tabular}{ll}
$\top(p_1)$ & if $e_0(p)=1$,\\
$\bot(p_1)$& if $e_0(p)=0$,\\
$p_j$ & if $e_0(p)=a_j$ (for $1 \leq j \leq n$)\\
\end{tabular}\right. $$

\

\noindent and let $\gamma(p_1,\ldots,p_n)=\sigma_0(\varphi)$. Observe that some of the variables $p_j$ may not appear in $\gamma$, but at least one variable $p_j$ (with $k+1 \leq j \leq n$) must occur in $\gamma$, by the hypothesis over $e_0$. Now we can state two immediate facts:
\begin{description}
\item[{\bf Fact 1:}] Given an evaluation $e$ for $L_1$, if $e(p_j)\in A_2$ for every $1 \leq j \leq n$ then
$e(\gamma) \in F_2$.
\end{description}
{\em Proof:}  follows from the observation at the beginning  of the proof, and by noting that  $\gamma$ is an instance of a tautology of $L_2$.
\begin{description}
\item[{\bf Fact 2:}]  Given an evaluation $e$ for $L_1$, if $e(p_j)=a_j$ for $1 \leq j \leq n$ then  $e(\gamma) = e_0(\varphi) \not\in F_1$.
\end{description}
{\em Proof:} Observe that, from the hypothesis, it follows that $e(\sigma_0(p_i))=e_0(p_i)$ for every $1 \leq i \leq m$. \\

Now, for any propositional variable $p$, let  $\alpha^j_j(p)=p$ for every $1 \leq j \leq n$, and let $\bigcirc(p)$ be the finite set of formulas
$$\bigcirc(p) = \{\gamma(\alpha^i_1(p),\ldots,\alpha^i_n(p)) \ : \  k+1 \leq i \leq n\}.$$
Let $e$ be an evaluation in $L_1$. Observe the following:\\[1mm]
(i) If $e(p)\in A_2$ then $e(\alpha^i_j(p)) \in A_2$ (since ${\bf A}_2$ is a subalgebra). For each $k+1 \leq i \leq n$ let $e_i$ be an evaluation for $L_1$ such that $e_i(p_j)= e(\alpha^i_j(p))$, for every $1 \leq j \leq n$. Then $e_i(\gamma) \in F_2$ , by Fact 1. But $e_i(\gamma) = e(\gamma(\alpha^i_1(p),\ldots,\alpha^i_n(p)))$  and so
$e(\gamma(\alpha^i_1(p),\ldots,\alpha^i_n(p))) \in F_2$
for every $k+1 \leq i \leq n$. This means that $e[\bigcirc(p)] \subseteq F_1$ if $e(p)\in A_2$. \\[2mm]
(ii) If $e(p)\notin A_2$ then $e(p)=a_i$ for some $k+1 \leq i \leq n$. From this, $e(\alpha^i_j(p))=a_j$ for all $1 \leq j \leq n$. Let $e'$ be an evaluation for $L_1$ such that $e'(p_j)=a_j$, for every $1 \leq j \leq n$. Then $e'(\gamma) = e(\gamma(\alpha^i_1(p),\ldots,\alpha^i_n(p)))$. But, by Fact 2, $e'(\gamma) =e_0(\varphi) \notin F_1$ and so $e(\gamma(\alpha^i_1(p),\ldots,\alpha^i_n(p))) \notin F_1$. Thus, $e[\bigcirc(p)] \not\subseteq F_1$ if $e(p) \notin A_2$. Equivalently, $e(p) \in A_2$ if $e[\bigcirc(p)] \subseteq F_1$. \\[1mm]
From the observations (i) and (ii) it follows that\\

\hspace{0.5cm} $(*) \hspace{2cm}e[\bigcirc(p)] \subseteq F_1 \ \mbox{ iff } \ e(p) \in A_2.$\\[2mm]
Finally, let $L_1^+$ be the logic obtained from $L_1$ by adding $\varphi$ (and all of its instances) as a theorem.
As observed above,
$$\Gamma \vdash_{L_1^+} \psi \ \ \mbox{ iff }  \Gamma, \{\sigma(\varphi) \ : \ \sigma \ \mbox{is a substitution in $\mathcal{L}(\Theta)$}\} \vdash_{L_1} \psi.$$

\begin{description}
\item[{\bf Fact 3:}]
Let $\Gamma \cup \{\psi\}$ be  a finite a set of formulas in $\mathcal{L}(\Theta)$ depending on the variables $p_1,\ldots,p_t$. Then \\[2mm]
$(**) \hspace{2cm}\Gamma \vdash_{L_2} \psi \ \ \mbox{ iff }  \Gamma, \bigcirc(p_1), \ldots, \bigcirc(p_t) \vdash_{L_1} \psi.$ 
\end{description}
{\em Proof:} Assume that $\Gamma \vdash_{L_2} \psi$ and let  $e \in Hom(\mathcal{L}(\Theta),{\bf A}_1)$ such that $e[\Gamma \cup \bigcup_{i=1}^t \bigcirc(p_i)] \subseteq F_1$. By $(*)$, $e(p_i) \in A_2$ for every $1 \leq i \leq t$. Consider now an evaluation $\bar e \in Hom(\mathcal{L}(\Theta),{\bf A}_2)$ such that $\bar e(p)=e(p)$ if $p \in \{p_1,\ldots,p_t\}$, and $\bar e(p)=0$ otherwise. Then $\bar e(\beta)=e(\beta)$ for every $\beta$ in $\mathcal{L}(\Theta)$ depending on the variables $p_1,\ldots,p_t$. Thus, $\bar e[\Gamma] \subseteq F_1 \cap A_2 = F_2$ whence $\bar e(\psi) \in F_2$, by hypothesis. That is, $e(\psi) \in F_1$ and so  $\Gamma, \bigcirc(p_1), \ldots, \bigcirc(p_t) \vdash_{L_1} \psi$.

Conversely, assume that $\Gamma, \bigcirc(p_1), \ldots, \bigcirc(p_t) \vdash_{L_1} \psi$ and consider an evaluation $\bar e\in Hom(\mathcal{L}(\Theta),{\bf A}_2)$ such that $\bar e[\Gamma] \subseteq F_2$. Define an evaluation $e \in Hom(\mathcal{L}(\Theta),{\bf A}_1)$ such that $e(p)=\bar e(p)$ for every variable $p$. Then $e(\beta)=\bar e(\beta)$ for every $\beta$ in $\mathcal{L}(\Theta)$ and so $e[\Gamma] \subseteq F_1$ and also $e[\bigcirc(p_i)] \subseteq F_1$ for every $1 \leq i \leq t$, by $(*)$. By hypothesis, $e(\psi) \in F_1$ and then $\bar e(\psi) \in F_1 \cap A_2$, that is, $\bar e(\psi) \in F_2$. This shows that  $\Gamma \vdash_{L_2} \psi$, proving Fact 3.\\

Consider now a finite a set of formulas   $\Gamma \cup \{\psi\}$ in $\mathcal{L}(\Theta)$ depending on the variables $p_1,\ldots,p_t$. Suppose that $\Gamma \vdash_{L_2} \psi$. Then $\Gamma, \bigcirc(p_1), \ldots, \bigcirc(p_t) \vdash_{L_1} \psi$, by Fact 3. But the latter implies that $\Gamma, \{\sigma(\varphi) \ : \ \sigma \ \mbox{is a substitution in $\mathcal{L}(\Theta)$}\} \vdash_{L_1} \psi$, because each $\bigcirc(p_i)$ is a set of instances of $\varphi$. From this, it follows that $\Gamma \vdash_{L_1^+} \psi$, by definition of $L_1^+$.

On the other hand,  suppose that $\Gamma \vdash_{L_1^+} \psi$. Given that ${\vdash_{L_1}} \subseteq {\vdash_{L_2}}$ (by  Lemma~\ref{IncMat})  and that $\vdash_{L_2} \varphi$ (by hypothesis) then $\Gamma \vdash_{L_2} \psi$, by definition of $L_1^+$. This shows that $L_1^+$ coincides with $L_2$ and so $L_1$ is maximal w.r.t.\ $L_2$.
\end{proof}

In the next example we show an application of Theorem~\ref{maxthm} in order to prove some maximality conditions for two logics related to the well-known 4-valued logic $\mathcal{FOUR}$ introduced by Belnap and Dunn \cite{Du76,Be76,Be77}. \\

\begin{figure}[h!] \label{m4}
   \centerline{ \includegraphics[width=0.5\textwidth]{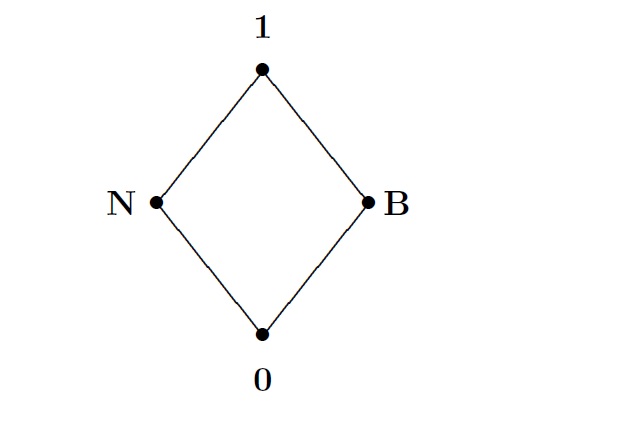}}
    \caption{Lattice $M_4$.}
\end{figure}

\begin{example}
Consider Belnap-Dunn's matrix logic  $\mathcal{BD}=\langle \mathfrak{M}_{4},\{1,B\}\rangle$, where  $\mathfrak{M}_{4}=\langle M_4, \land, \lor, \neg\rangle$ is the algebra associated to the {\em logical lattice} $M_4$ (see Fig.~1) expanded with the De Morgan negation $\neg$ defined as:

\begin{center}
\begin{tabular}{|c||c|} \hline
$\quad$ & $\neg$ \\
 \hline \hline
    $1$   & 0    \\ \hline
     $B$   &  $B$    \\ \hline
     $N$   &  $N$    \\ \hline
     $0$   & $1$    \\ \hline
\end{tabular}
\end{center}

Much later, De and Omori considered in~\cite{DO2015} the expansion $\mathcal{BD}^{\sim}$ of $\mathcal{BD}$ by adding the strong negation $\sim$, given by the following table:\\

\begin{center}
\begin{tabular}{|c||c|}         \hline
$x$ & ${\sim}x$ \\ \hline
0       & 1          \\
N       & B          \\
B       & N          \\
1       & 0          \\ \hline
\end{tabular}
\end{center}

\noindent  On the other hand, before Belnap and Dunn's investigations,  L. Monteiro  already considered in 1963  (see~\cite{LMonteiro}) the 4-valued algebra  $\mathfrak{M}_{4m}$ obtained from  $\mathfrak{M}_{4}$  by adding a modal operator $\square$  defined as follows:

\begin{center}
\begin{tabular}{|c||c|} \hline
$\quad$ & $\square$ \\
 \hline \hline
    $1$   & 1    \\ \hline
     $B$   & $0$    \\ \hline
     $N$   & $0$    \\ \hline
     $0$   & $0$    \\ \hline
\end{tabular}
\end{center}

This led to  A. Monteiro to consider the variety $\mathbf{TMA}$ of {\em tetravalent modal algebras}, which is the one generated by  $\mathfrak{M}_{4m}$ (cf.~\cite{Loureiro}).  As  proven by Font and Rius in~\cite{FR2}, the (degree-preserving) logic of $\mathbf{TMA}$ is characterized by the matrix logic ${\cal M}_B =\langle\mathfrak{M}_{4m},\{B, 1\} \rangle$.
Previous to \cite{DO2015} and with a different motivation, Coniglio and Figallo define in~\cite{CF2014} the logic ${\cal M}_B^{\sim} =\langle\mathfrak{M}_{4m}^{\sim},\{B, 1\} \rangle$, the expansion of ${\cal M}_B$ with the strong negation $\sim$ described above,  characterizing the (degree-preserving) logic of the variety generated by $\mathfrak{M}_{4m}^{\sim}$  (which was independently introduced by A. Monteiro in~\cite{Monteiro:69} and by G. Moisil in~\cite{Ml72}.)

By using Theorem~\ref{maxthm}, it is easy to show that both  ${\cal M}_B^{\sim}$ and $\mathcal{BD}^{\sim}$ are maximal relative to {\sf CPL} presented in the signature $\Theta=\{\land,\lor,\neg\sim,\square\}$ and $\Theta'=\{\land,\lor,\neg\sim\}$ over the two-element Boolean algebra $\mathfrak{B}_2$, respectively (where $\square p$ is equivalent to $p$ and  $\neg p$ is equivalent to  ${\sim}p$). Indeed, observe that ${\bf B}_2$ (expanded by $\sim$ and $\square$) is a subalgebra of $\mathfrak{M}_{4m}^{\sim}$, and $\top(p)=p \lor {\sim} p$ and $\bot(p)=p \land {\sim}p$ are as required. Notice that, since there are in $M_4$ just two values besides the `classical' ones, namely $a_1=N$ and $a_2=B$, the formulas $\alpha^1_2(p)=\alpha^2_1(p)={\sim}p$ are such that $e(\alpha^1_2(p))=B$ if $e(p)=N$, $e(\alpha^2_1(p))=N$ if $e(p)=B$. Therefore, it follows from Theorem~\ref{maxthm} that ${\cal M}_B^{\sim}$  is maximal reative to {\sf CPL} presented over the signature $\Theta$. Similarly, it also follows that $\mathcal{BD}^{\sim}$ is maximal relative to {\sf CPL} presented over the signature $\Theta'$ (the latter corresponding to~\cite[Theorem~3]{DO2015}).  \hfill $\blacksquare$
\end{example}

As an immediate consequence of Theorem~\ref{maxthm}, it follows that any 3-valued logic which extends {\sf CPL} and it can express the top and the bottom formulas, is maximal w.r.t.\ {\sf CPL}.

\begin{corollary} \label{3val-max}
Let ${\bf A}_1$ be an algebra defined over a signature $\Theta$ with domain $A_1=\{0,1/2,1\}$, and consider the matrix logic  $L_1=\langle {\bf A}_1, F_1\rangle$ where $0 \not\in F_1$ and $1 \in F_1$. Further, let ${\bf A}_2$ be a subalgebra of ${\bf A}_1$, with $A_2 = \{0, 1\}$, and assume that the matrix logic $L_2=\langle {\bf A}_2, \{1\}\rangle$ is a presentation of classical propositional logic {\sf CPL} over signature $\Theta$ such that $L_2$ is distinct from $L_1$.
Suppose additionally there are formulas $\top(p)$ and $\bot(p)$ in $\mathcal{L}(\Theta)$ on one variable $p$ such that $e(\top(p))=1$ and $e(\bot(p))=0$, for every evaluation $e$ for $L_1$. Then,  $L_1$ is maximal w.r.t.\ {\sf CPL} (presented as $L_2$).
\end{corollary}
\begin{proof}
Observe that $L_1$ and $L_2$  are  matrix logics as in Lemma~\ref{IncMat}, since $\{1\}=F_1 \cap A_2$.
Given that $A_1$ contains just one element out of $ \{0,1\}$, namely $a_1=1/2$, then Theorem~\ref{maxthm} can be applied (since requirement~(3) is satisfied by vacuity). As a consequence of Theorem~\ref{maxthm}, $L_1$ is maximal w.r.t.\ {\sf CPL} (presented as $L_2$).
\end{proof}

In the next example some instances of Corollary~\ref{3val-max} are analyzed, showing the strength of this result: indeed, several well-known 3-valued logics which are known to be maximal w.r.t.\ {\sf CPL} fall inside the scope of  Corollary~\ref{3val-max}. 

\begin{example} \label{ExL3}
(1) Let us begin with \L ukasiewicz 3-valued logic $\L_3=\langle \algLfin_3,\{1\}\rangle$, where $\algLfin_3$ is the usual 3-valued algebra for $\L_3$ over $\Theta=\{\neg,\to\}$ with domain $\{0,1/2,1\}$.
 Let $\mathsf{L}_1^1=\langle {\bf B}_2,F\rangle$ be a presentation of {\sf CPL}, where  ${\bf B}_2$ is  the two-element Boolean algebra over $\Theta$ with domain $\{0,1\}$ and $F=\{1\}$. It is easy to see that $\L_3$ satisfies the requirements of Corollary~\ref{3val-max} by taking $\top(p)=(p \to p)$ and $\bot(p)=\neg (p \to p)$. This produces a new proof of the maximality of $\L_3$ w.r.t.\ {\sf CPL}. In order to illustrate this fact consider by  instance  $\varphi(p_1)=p_1 \vee \neg p_1 := (p_1 \to \lnot p_1)\to \lnot p_1$, a formula which is valid in {\sf CPL} but it is not valid in  $\L_3$. Indeed, any evaluation $e_0$ in $\L_3$ where $e_0(p_1)=1/2$ is such that $e_0(\varphi)=1/2$, a non-designated truth-value. By following the construction described in the proof of Theorem~\ref{maxthm} (where $\alpha^1_1(p)=p$), it follows that $\gamma(p_1)=\varphi(p_1)$, and so $\circ(p)=p \vee \neg p$ is a recovery operator for $\L_3$ w.r.t.\ {\sf CPL} defined in terms of  $\varphi$. Thus, $\L_3$ plus $\varphi$ coincides with {\sf CPL}. Notice that the truth-table of the recovery operator $\circ$ is as follows:
$$
\begin{array}{|c||c|} \hline
   & \circ  \\ \hline \hline
  1 & 1   \\ \hline
  1/2 & 1/2  \\ \hline
  0 & 1  \\ \hline
\end{array}
$$
(2) Consider now the logic $\mathsf{L}_2^1=\langle \algLfin_3,\{1, 1/2\}\rangle$. As it is well known, the matrices of $\L_3$ are functionally equivalent to that of the 3-valued paraconsistent logic ${\sf J}_3$, introduced by da Costa and  D'Ottaviano, see \cite{dot:dac:70}. This means that  $\mathsf{L}_2^1$ coincides with ${\sf J}_3$ up to language. By item~(1) and  Corollary~\ref{3val-max} it follows that  $\mathsf{L}_2^1$ is maximal w.r.t.\ {\sf CPL}.  This constitutes a new proof of the maximality of ${\sf J}_3$ (and all of its alternative presentations, such as {\bf LFI1} or {\bf MPT}, see~\cite{Con:Sil}) w.r.t.\ {\sf CPL}. A generalization of ${\sf J}_3$ to $\L_4$, called ${\sf J}_4$, will be proposed in Subsection \ref{sectJ4}. As an illustration  of how the technique of the proof works, let $\varphi(p_1)=\neg((\neg p_1 \to p_1) \wedge (p_1 \to  \neg p_1))$. It is easy to see that    $\varphi(p_1)$ is valid in {\sf CPL} but it is not valid in    $\mathsf{L}_2^1$: any evaluation $e_0$ in $\mathsf{L}_2^1$ with $e_0(p_1)=1/2$ is such that $e_0(\varphi)=0$. Then, by the proof of Theorem~\ref{maxthm} (where $\alpha^1_1(p)=p$),  $\gamma(p_1)=\varphi(p_1)$ and so $\circ(p)= \neg((\neg p \to p) \wedge (p \to  \neg p))$ is a recovery operator  for   $\mathsf{L}_2^1$ w.r.t.\ {\sf CPL} defined in terms of instances of $\varphi$. This means that  $\mathsf{L}_2^1$ plus $\varphi$ coincides with {\sf CPL}.  The truth-table of $\circ$ is as follows:
$$
\begin{array}{|c||c|} \hline
   & \circ  \\ \hline \hline
  1 & 1  \\ \hline
  1/2 &  0  \\ \hline
  0 & 1  \\ \hline
\end{array}
$$
(3) In an unpublished draft, J. Marcos~\cite{Marcos} (see also~\cite[Section 5.3]{car:con:mar:07}) proposes a family of 8,192 logics which are 3-valued and paraconsistent, belonging to the hierarchy of {\bf LFI}s. Among these logics, there is ${\sf J}_3$  (whose truth-tables can define the matrices of all the other logics in the family) and Sette's logic $\mathsf{P}^1$ (see~\cite{Sette}), whose truth-tables are definable by the matrices of any of the logics in the family. All these logics are maximal w.r.t.\ {\sf CPL} presented in the signature $\{\land,\lor,\to,\neg,\circ\}$ such that $\circ\varphi$ is valid for every $\varphi$ (that is, algebraically, $\circ(x)=1$ for all  $x \in \{0,1\}$). The proof of maximality of all these logics w.r.t.\ {\sf CPL} follows easily from Corollary~\ref{3val-max}  by taking $\top(p)=p \to p$ and $\bot(p)=p \land \neg p \land \circ p$.\\[1mm]
(4) Let $\mathsf{I}^1$ be the 3-valued paracomplete logic introduced by A.M. Sette and W.A. Carnielli in~\cite{SetCar}. It is defined over $\Theta=\{\to,\neg\}$ with domain $\{0,1/2,1\}$ and designated value 1, and whose operations are given by the tables below.

$$
\begin{array}{|c||c|c|c|}  \hline
  \to & 1 & 1/2  & 0\\ \hline \hline
  1 & 1 & 0 & 0 \\ \hline
  1/2 & 1 &  1  & 1\\ \hline
  0 & 1 & 1  & 1 \\ \hline
\end{array}
\hspace{1 cm}
\begin{array}{|c||c|} \hline
   & \neg   \\ \hline \hline
  1 & 0   \\ \hline
  1/2 & 0  \\ \hline
  0 & 1  \\ \hline
\end{array}
$$

\noindent Once again, the maximality of $\mathsf{I}^1$ w.r.t.\ {\sf CPL} follows from Corollary~\ref{3val-max} by taking $\top(p)=p \to p$ and $\bot(p)=\neg(p \to p)$.\\[1mm]
(5) Let $G_{n+1}=\langle {\bf G}_{n+1},\{1\}\rangle$ be the $(n+1)$-valued G\"{o}del logic defined over the algebra ${\bf G}_{n+1}$ for $\Theta=\{\land,\lor,\to,\neg\}$ with domain $\big\{0,\frac{1}{n},\dots,\frac{n-1}{n},1\big\}$ such that $x \land y = \min \{x,y\}$; $x \lor y = \max \{x,y\}$; $x \to y= 1$ if $x \leq y$ and $x \to y= y$ otherwise; and $\neg x =1$ if $x=0$ and $\neg x=0$ otherwise. In particular, $G_3$ is defined over $\{0,1/2,1\}$ with the following tables for $\to$ and $\neg$:

$$
\begin{array}{|c||c|c|c|}  \hline
  \to & 1 & 1/2  & 0\\ \hline \hline
  1 & 1 & 1/2 & 0 \\ \hline
  1/2 & 1 &  1  & 0\\ \hline
  0 & 1 & 1  & 1 \\ \hline
\end{array}
\hspace{1 cm}
\begin{array}{|c||c|} \hline
   & \neg   \\ \hline \hline
  1 & 0   \\ \hline
  1/2 & 0  \\ \hline
  0 & 1  \\ \hline
\end{array}
$$

\noindent Clearly $G_3$ falls within the scope of Corollary~\ref{3val-max} (where  $\top(p)=p \to p$ and $\bot(p)=p \land \neg p$) and so it is maximal w.r.t.\ {\sf CPL} presented over $\Theta$. Observe that  for $n \geq 3$ the algebra ${\bf G}_{n+1}$ does not have enough expressive power to define all the formulas $\alpha^i_j$ in order to apply Theorem~\ref{maxthm}. For instance, in ${\bf G}_{4}$ there are no formulas $\alpha^1_2(p)$ and $\alpha^2_1(p)$ such that $e(\alpha^1_2(p))=2/3$ if $e(p)=1/3$ and  $e(\alpha^2_1(p))=1/3$ if $e(p)=2/3$. \hfill $\blacksquare$
\end{example}

\begin{example} \label{EjIdeal} In \cite{ArieliAZ11a} the authors introduced the notion of {\em ideal paraconsistent logics}.  Together with this, they presented a family $\mathcal{M}_{n+2}$ of $(n+2)$-valued matrix logics (with $n \geq 2$) which are ideal paraconsistent (and so, from the very definition, they are also maximal w.r.t. ${\sf CPL}$, see Definition~\ref{IdPar} in Section~\ref{sectIdeal}). The fact that  all these logics are maximal w.r.t. $\sf CPL$ (as proved in  \cite{ArieliAZ11a}) can also be proved by applying Theorem~\ref{maxthm}, as it will be shown in what follows.

Given $n\geq 2$ consider the algebras $\mathbf{A}_{n+2}$ over the signature $\Theta=\{\neg,\diamond,\supset\}$ with domain $A_{n+2}=\{0,1,a_1,\ldots,a_n\}$ such that the operations are defined as follows: $\neg 0=1$, $\neg 1=0$ and $\neg x=x$ otherwise; $\diamond 0=1$,  $\diamond 1=0$,  $\diamond a_i=a_{i+1}$ if $i<n$ and  $\diamond a_n=a_1$; $x \supset y = 1$ if   $x \notin D=\{1,a_1\}$ and $x \supset y = y$ otherwise. The  logic $\mathcal{M}_{n+2}$ is defined by the logical matrix $\langle\mathbf{A}_{n+2},D\rangle$ for every $n\geq 2$. Let us see that the conditions of Theorem~\ref{maxthm} are satisfied for every logic $\mathcal{M}_{n+2}$ w.r.t. $\sf CPL$. It is easy to see that $\{0,1\}$  is a subalgebra of $\mathbf{A}_{n+2}$ and so, by Lemma \ref{IncMat}, $\mathcal{M}_{n+2}$ is a sublogic of $\sf CPL$ presented in the signature $\Theta$ in which $\diamond$ coincides with negation and  $1$ is the designated value. In addition, it is easy to see that, given a propositional variable $p$, the formulas $\top(p)=(p \supset \diamond p) \supset (p \supset \diamond p)$ and $\bot(p)=\neg\top(p)$ are such that $e(\top(p))=1$ and $e(\bot(p))=0$, for every evaluation $e$. Consider now the formulas $\alpha^i_j(p)=\diamond^{j-i}p$ if $i<j$ and    $\alpha^i_j(p)=\diamond^{n-i+j}p$ if $i>j$, where $\diamond^{0}p=p$ and $\diamond^{i+1}p=\diamond\diamond^i p$, for every $i$. An easy computation shows that $e(\alpha^i_j(p))=a_j$ if $e(p)=a_i$, for every $i\neq j$. Therefore, the conditions of  Theorem~\ref{maxthm} are fullfilled and so each logic  $\mathcal{M}_{n+2}$ is maximal w.r.t. $\sf CPL$.

The question of ideal paraconsistent logics in the present framework will be treated again in Section \ref{sectIdeal}.\hfill $\blacksquare$
\end{example}

The examples given above show the value of  Theorem~\ref{maxthm} in order to state maximality of logics under certain hypothesis concerning the expressive power of the given logics. Indeed,  several proofs of maximality found in the literature can be easily obtained as a consequence of Theorem~\ref{maxthm}: for instance, the ones given for the 3-valued paraconsistent logic $\mathsf{P}^1$ in \cite[Proposition 11]{Sette}, for the 3-valued logic  $\mathsf{I}^1$  in \cite[Proposition 17]{SetCar} and for ${\sf J}_3$ (formulated as the equivalent logic {\sf LFI1}) in  \cite[Theorem 4.6]{CarMarAmo}, respectively. It is worth noting that all the examples of maximality of a logic $L_1$ w.r.t. another logic $L_2$ given in this section, as well as the examples to be given in the rest of the paper, are non-vacuous in the sense of  Remark~\ref{vacuous}. Indeed, in all the examples of maximality presented here the set of theorems of $L_1$ is strictly contained in the set of theorems of $L_2$, thus the notion of maximality holds in a non-trivial way. For instance,   the formula $p \to \circ p$ is a theorem of {\sf CPL} which does not hold in any of the logics presented in Example~\ref{ExL3}(3), while the formula $p \to \neg\diamond p$ is a theorem of {\sf CPL} which does not hold in any of the systems $\mathcal{M}_{n+2}$ presented in Example~\ref{EjIdeal}.
 On the other hand, it should be observed that the set of designated values may not play a relevant role with respect to maximality, for instance, when analyzing maximality with respect to $\sf CPL$ (recall e.g. Corollary \ref{3val-max} or see Proposition \ref{maxLn} in next section).

As observed above,  Theorem~\ref{maxthm} cannot be applied to  logics which do not have enough expressive power, as seen in Examples~\ref{ExL3}(5) for G\"{o}del's logics $G_n$ (with $n \geq 4$). This is not the case for finite-valued {\L}ukasiewicz logics, as it will be shown in  the next section.

\section{Maximality between finite-valued {\L}ukasiewicz logics induced by order filters}
\label{Sect-max}

In the rest of the paper we will deal with matrix logics based on the family of finite-valued {\L}ukasiewicz logics $\L_n$ with $n \geq 2$.
The $(n+1$)-valued {\L}ukasiewicz logic can be semantically defined as the matrix logic
$$\L_{n+1}=\langle \algLfin_{n+1}, \{1\}  \rangle, $$
where {$\algLfin_{n+1} = (\domLfin_{n+1}, \neg, \to)$} with $\domLfin_{n+1} = \big\{0,\frac{1}{n},\dots,\frac{n-1}{n},1\big\}$, and the operations are defined as follows: for every $x,y \in \domLfin_{n+1}$,
\begin{itemize}
\item[] $\neg x =1-x$ (\L ukasiewicz negation)
\item[] $x \to y = \min \{1, 1-x+y\}$ (\L ukasiewicz implication)
\end{itemize}

\noindent The following operations can be defined in every algebra $\algLfin_{n+1}$:

\begin{itemize}
\item[] $x \otimes y=\neg(x \to \neg y)=\max \{0, x+y-1\}$ (strong conjunction)
\item[] $x \oplus y=\neg x \to y=\min \{1, x+y\}$ (strong disjunction)
\item[] $x \vee y=(x \to  y) \to y=\max \{x,y\}$ (lattice disjunction)
\item[] $x \wedge y=\neg((\neg x \to  \neg y) \to \neg y)=\min \{x,y\}$ (lattice conjunction)
\end{itemize}

Observe that $\L_2$ is the usual presentation of  classical propositional logic {\sf CPL}  as a matrix logic over the two-element Boolean algebra ${\bf B}_2$ with domain $\{0,1\}$ with signature $\{\neg,\to\}$.

The logics $\L_{n}$ can also be presented as Hilbert calculi that are axiomatic extensions of the infinite-valued {\L}ukasiewicz logic \L$_\infty$.  Recall that $\L_{\infty}$ is algebraizable and the class $MV$ of all MV-algebras is its equivalent quasivariety semantics \cite{RTV90,CiMuOt99}. Since algebraizability is preserved by finitary extensions then every finite valued $\L$ukasiewicz logic $\L_n$ is also algebraizable, and we will denote by $MV_n$ its corresponding subvariety of algebras.

In this section,  finite-valued {\L}ukasiewicz logics with a set of designated values possibly different to $\{ 1\}$ will be studied from the point of view of maximality.  First, some notation will be introduced.

For every $i\geq 1$ and for every $x \in\domLfin_{n+1}$, $i x$ will stand for $x\oplus\dots\oplus x$ ($i$-times), while $x^i$ will stand for $x\otimes\dots\otimes  x$ ($i$-times).

For $1 \leq i \leq n$ let
$$F_{i/n} =\{x \in \domLfin_{n+1} \ : \  x \geq i/n\}=\big\{\frac{i}{n},\dots,\frac{n-1}{n},1\big\}$$
be the order filter generated by $i/n$, and let
$$\mathsf{L}^i_n=\langle \algLfin_{n+1}, F_{i/n}  \rangle$$
be the corresponding matrix logic.
From now on,  the consequence relation of $\mathsf{L}^i_n$ is denoted by $\vDash_{\mathsf{L}^i_n}$.
Observe that $\L_{n+1}= \mathsf{L}^n_n$ for every $n$. In particular, {\sf CPL} is $\mathsf{L}^1_1$ (that is, $\L_2$). If  $1 \leq i, m \leq n$, we can also consider the following matrix logic:
$$\mathsf{L}^{i/n}_m=\langle \algLfin_{m+1}, F_{i/n} \cap \domLfin_{m+1}  \rangle.$$
Since $ F_{i/n} \cap \domLfin_{m+1} = F_{j/m}$ for some  $1\leq j\leq m$, $\mathsf{L}^{i/n}_m=\mathsf{L}^{j}_{m}$ for that $j$. It is interesting to notice that some of these logics are paraconsistent, and some are not. Indeed, it is easy to prove the following characterization.

\begin{proposition} \label{parLqi} The logic $\mathsf{L}^i_n$ is paraconsistent w.r.t.\ $\neg$ iff $i/n \leq 1/2$.
\end{proposition}
\begin{proof}
$\mathsf{L}^i_n$ is paraconsistent w.r.t.\ $\neg$ iff there exists $x \in \domLfin_{n+1}$ such that $x \geq i/n$ and $\neg x \geq i/n$, iff $i/n \leq x \leq (n-i)/n$ for some  $x \in \domLfin_{n+1}$, iff  $i/n \leq  (n-i)/n$, iff $2i \leq n$.
\end{proof}

Thus, for instance, for $n=5$  it follows that $\mathsf{L}_5^1$ and $\mathsf{L}_5^2$ are paraconsistent, while $\mathsf{L}_5^3$, $\mathsf{L}_5^4$ and $\mathsf{L}_5^5=\L_6$ are  explosive. By its turn, if  $n=3$ then $\mathsf{L}_3^1$  is paraconsistent, while $\mathsf{L}_3^2$ and $\mathsf{L}_3^3=\L_4$ are  explosive. The paraconsistent logics of this form  which are maximal w.r.t. {\sf CPL}  will be analyzed with more detail in Section \ref{sectIdeal}.

Theorem \ref{maxthm} can be  used in order to analyze the maximality of  the logic $\mathsf{L}^i_n$ w.r.t.\ $\mathsf{L}^{i/n}_m$ whenever $m | n$ (taking into account that $\algLfin_{m+1}$ is a subalgebra of $\algLfin_{n+1}$ iff $m|n$). In particular, the maximality of  certain instances of $\mathsf{L}^i_n$ w.r.t.\ {\sf CPL} can be obtained by using Theorem \ref{maxthm}.

The following examples  deal with the algebras $\algLfin_n$ which, as observed above, can define a meet operator $\wedge$ such that, for any order filter $F$, $(a \wedge b) \in F$ iff $a,b \in F$. Because of this, a recovery operator $\circ(p)$ will be considered instead of a recovery set $\bigcirc(p)$, consisting of the conjunction of all of its members.

\begin{example} Let us first consider the case of $\algLfin_4$.
For $1 \leq i \leq 3$ let $\mathsf{L}^i_3=\langle \algLfin_4,F_{i/3}\rangle$. Then $F_{1/3}=\{1/3, 2/3,1\}$, $F_{2/3}=\{2/3,1\}$ and  $F_{3/3}=F_1=\{1\}$.  As in the previous example, it can be proved that  each $\mathsf{L}^i_3$ satisfies the requirements of Theorem~\ref{maxthm} w.r.t.\ {\sf CPL} and so each $\mathsf{L}^i_3$ is maximal w.r.t.\ {\sf CPL}, presented as ${\sf CPL}=\langle {\bf B}_2,\{1\}\rangle$. Indeed, ${\bf B}_2$ is a subalgebra of $\algLfin_4$ and $\top(p)=(p \to p)$ and $\bot(p)=\neg (p \to p)$ are as required.   Finally, the formulas $\alpha^1_2(p)=\alpha^2_1(p)=\neg p$ are such that $e(\alpha^1_2(p))=2/3$ if $e(p)=1/3$, $e(\alpha^2_1(p))=1/3$ if $e(p)=2/3$.  (Observe that there are in $\algLfin_4$ just two `non-classical' values: $a_1=1/3$ and $a_2=2/3$.)

Fix $1 \leq i \leq 3$. Thus, given a theorem $\varphi(p_1,\ldots,p_m)$ of {\sf CPL} which is not valid in $\mathsf{L}^i_3$, consider the formula $\gamma(p_1,p_2)$ as in the proof of Theorem~\ref{maxthm}. Then, the formula $\circ(p)=\gamma(p,\neg p) \wedge \gamma(\neg p,p)$ defines an operator (in terms of a conjunction of instances of $\varphi$) which allows to recover classical logic inside $\mathsf{L}^i_3$. \hfill $\blacksquare$
\end{example}


From Komori's characterization of axiomatic extensions of (infinite-valued) \linebreak {\L}ukasiewicz logic ${\L}_{\infty}$ \cite{Komori:SuperLukasiewiczPropositional}, it directly follows that the logic $\L_{n+1}$ is maximal w.r.t.\ {\sf CPL} iff $n$ is a prime number. By adapting our previous arguments, we can obtain the following extension of  this classical result for logics matrix logics over $\algLfin_{n+1}$ with (almost) arbitrary filters.

\begin{proposition} \label{maxLn} Let  $n\geq 2$ and $\emptyset\neq F \subseteq \domLfin_{n+1}$. Then, the logic $L=\langle\algLfin_{n+1},F\rangle$ is maximal w.r.t.\ {\sf CPL} provided that $0 \notin F$ and $n$ is a prime number.
\end{proposition}

Observe that, as a direct consequence, all the logics $\mathsf{L}^i_q$ with $q$ prime are maximal w.r.t.\ classical logic.

\begin{corollary} \label{maxLqi} Let  $q$ be a prime number, and $1 \leq i \leq q$. Then, the logic $\mathsf{L}^i_q$ is maximal w.r.t.\ {\sf CPL}.
\end{corollary}

\begin{remark} \label{Lqi-indist}
Note that, for a given prime $q$, if $i < j$ the set of theorems of $\mathsf{L}^j_q$ is strictly included in the set of theorems of $\mathsf{L}^i_q$. However this does not contradict the fact that both logics are maximal w.r.t.\ {\sf CPL}, since their consequence relations are in fact incomparable.  For example, the set of theorems of $\mathsf{L}^2_3$ is included in the set of theorems of $\mathsf{L}^1_3$, but the inclusion is strict: $\models_{\mathsf{L}^1_3} (p \vee \neg p)\otimes (p \vee \neg p)$ while $\not\models_{\mathsf{L}^2_3} (p \vee \neg p)\otimes (p \vee \neg p)$. It suffices to consider an evaluation $e$ such that $e(p) = 1/3$; then $e((p \vee \neg p)\otimes (p \vee \neg p)) = 1/3 \not\geq 2/3$. On the other hand, $\mathsf{L}^2_3$ is not a sublogic of $\mathsf{L}^1_3$: $p \models_{\mathsf{L}^2_3} (p \otimes p) \oplus (p \otimes p)$ but $p \not\models_{\mathsf{L}^1_3} (p \otimes p) \oplus (p \otimes p)$. In order to see this, consider an evaluation $e$ such that $e(p) = 1/3$; then $e((p \otimes p) \oplus (p \otimes p)) = 0$.
\end{remark}

Next examples exploit the fact that $\L_{n+1}$  is a sublogic of $\L_{m+1}$ iff $m$ divides $n$, considering additional filters as designated values, and obtaining maximality in some cases.

\begin{example} Now, the logics asociated to the algebra $\algLfin_5$ will be analyzed.
For $1 \leq i \leq 4$ let
$\mathsf{L}_4^i=\langle \algLfin_5,F_{i/4}\rangle$ such that $F_{1/4}=\{1/4,1/2,3/4,1\}$,  $F_{2/4}=F_{1/2}=\{1/2,3/4,1\}$, $F_{3/4}=\{3/4,1\}$,  and $F_{4/4}=F_1=\{1\}$. Since $2$ divides $4$ then $\algLfin_3$ is a subalgebra of $\algLfin_5$ and $\L_5$ is a sublogic of $\L_3$. We will prove that, indeed, any $\mathsf{L}_4^i$ (for $1 \leq i \leq 4$) is maximal w.r.t.\ $\mathsf{L}_2^{i/4}=\langle \algLfin_3,F_{i/4} \cap \domLfin_{3}\rangle$, by using Theorem~\ref{maxthm}.

By Lemma~\ref{IncMat}, each $\mathsf{L}_4^i$ is a sublogic of  $\mathsf{L}_2^{i/4}$. $\algLfin_3$ is a subalgebra of  $\algLfin_5$ and $\top(p)=(p \to p)$ and $\bot(p)=\neg (p \to p)$ are as required.  Let $a_1=1/2$, $a_2=1/4$ and $a_3=3/4$, and consider the formulas $\alpha^2_1(p)= p \oplus p$, $\alpha^2_3(p)=\alpha^3_2(p)=\neg p$, and  $\alpha^3_1(p)=p \otimes p$.  Finally, let $\alpha^i_i(p)= p$ for $i=2,3$. Then, the formulas $\alpha^i_j$ defined above are such that $e(\alpha^i_j(p))=a_j$ if $e(p)=a_i$, for $i=2,3$ and $j=1,2,3$.

Fix $1 \leq i \leq 4$. Thus, given a theorem $\varphi_i(p_1,\ldots,p_{m_i})$ of $\mathsf{L}_2^{i/4}$ which is not valid in $\mathsf{L}_4^i$, consider the formula $\gamma_i(p_1,p_2,p_3)$ as in the proof of Theorem~\ref{maxthm}. Then, the formula
$$\circ_i(p)=\gamma_i(p\oplus p,p, \neg p) \wedge \gamma_i(p\otimes p,\neg p,p)$$
defines a recovery operator (in terms of a conjunction of instances of $\varphi_i$) which allows to recover $\mathsf{L}_2^{i/4}$ inside $\mathsf{L}_4^i$. This shows that the latter is maximal w.r.t.\ the former.   \hfill $\blacksquare$
\end{example}

\begin{example}
Consider  now the case of $\algLfin_7$. Since 2 and 3 divide 6, it follows that $\algLfin_3$ and $\algLfin_4$ are subalgebras of $\algLfin_7$ and so $L=\langle \algLfin_7,F\rangle$ is a sublogic of both $\langle \algLfin_3,F\cap \domLfin_3 \rangle$ and $\langle\algLfin_4,F\cap \domLfin_4 \rangle$ for any non-trivial filter $F$ of $\algLfin_7$. However, it is not possible to prove the maximality of $L$ by applying Theorem~\ref{maxthm} since, for every formula $\alpha(p)$ and every evaluation $e$ in $\algLfin_7$, $e(\alpha(p))\neq 1/2$ if $e(p) \in \{1/3,2/3\}$ (since $\algLfin_4$ is a subalgebra), while  $e(\alpha(p))\notin \{1/3,2/3\}$ if $e(p)=  1/2$ (since $\algLfin_3$ is a subalgebra).   \hfill $\blacksquare$
\end{example}

As another example of application of Theorem \ref{maxthm}, we can obtain the following maximality condition of a logic $\mathsf{L}^i_n$ with respect to a logic ${\sf L}_{m}^{i/n}$.

\begin{proposition}\label{maxLnLm} Let $1 \leq i,m \leq n$. Then $\mathsf{L}^i_n=\langle \algLfin_{n+1},F_{i/n}\rangle$ is maximal w.r.t.\  $\mathsf{L}_{m}^{i/n}=\langle \algLfin_{m+1}, F_{i/n} \cap \domLfin_{m+1}
\rangle$ if the following condition
 holds:  there is some prime number $q$ and $k \geq 1$ such that $n=q^k$, and  $m=q^{k-1}$.
\end{proposition}

\begin{proof}
We recall that  $\algLfin_{n+1}$ is singly generated by any element $0<\frac{l}{n}<1$ such that $l$ and $n$ are mutually prime \cite[Lemma 1.2]{GpT98}. Then,   since $q$ is prime, $\domLfin_{q^{k}+1}\smallsetminus \domLfin_{q^{k-1}+1}=\{0<\frac{r}{q^{k}}<1  \ : \ r \mbox{ and } q \mbox{ are mutually prime}\}$ and therefore all conditions of Theorem \ref{maxthm} are satisfied.
\end{proof}

\section{On strong maximality and explosion rules in the logics $\mathsf{L}^i_q$} \label{sectLqi}

Along this section $q$ will denote a prime number.

In the previous section we have seen that  all the logics of the form $\mathsf{L}^i_q=\langle \algLfin_{q+1}, F_{i/q}  \rangle$ are maximal w.r.t. {\sf CPL}.
However, there are maximal logics that  are not maximal w.r.t. {\sf CPL} in  an stronger sense, as firstly considered in \cite{AvronAZ10,ArieliAZ11a} in the context of paraconsistent logics, or in \cite{RibCon12} in a more general context of belief revision techniques for change of logics.

\begin{definition}
Let $L_1$  and $L_2$ two standard propositional logics defined over the same signature $\Theta$ such that $L_1$ is a proper sublogic of $L_2$, i.e.\  such that ${\vdash_{L_1}} \subsetneq {\vdash_{L_2}}$. Then, $L_1$ is said to be {\em  strongly maximal} w.r.t.\ $L_2$ if, for every finitary rule $\varphi_{1}, \ldots , \varphi_{n}/ \psi$  over $\Theta$, if $\varphi_{1}, \ldots , \varphi_{n}\vdash_{L_2} \psi$ but $\varphi_{1}, \ldots , \varphi_{n}\nvdash_{L_1} \psi$, then the logic $L_1^*$  obtained from  $L_1$   by adding $\varphi_{1}, \ldots , \varphi_{n}/ \psi$ as structural rule, coincides with $L_2$.
\end{definition}

By $L_1^*$  above we mean the logic whose consequence relation $\vdash_{L_1^*}$  is the minimal extension of $\vdash_{L_1}$ such that
 $\sigma(\varphi_{1}), \ldots , \sigma(\varphi_{n}) \vdash_{L_1^*} \sigma(\psi)$ for any substitution $\sigma$ over  $\Theta$ (see e.g. \cite{Woj88,ArieliAZ11a}).

For instance, as observed in~\cite[Remark~14]{DO2015}, the logic $\mathcal{BD}^{\sim}$ introduced in Section \ref{recovery}, that is maximal w.r.t.\ {\sf CPL}, it is not strongly maximal relative to {\sf CPL}.
Thus, a natural question is whether a given logic is strongly maximal w.r.t.\ another logic. In particular, in this section, we are interested in studying the status of the logics $\mathsf{L}^i_q=\langle \algLfin_{q+1}, F_{i/q}  \rangle$ with $q$ prime in relation to the notion of strong maximality w.r.t.\ {\sf CPL}. We will show that the answer is negative, as each of them admits a proper extension by a finitary rule related to the explosion law w.r.t.\ {\L}ukasiewicz negation. In fact, in Section \ref{CPL}  it will be shown that such proper extensions are strongly maximal w.r.t.\ {\sf CPL}.

\begin{remark} \label{axiomHqi}
By using the techniques presented in \cite{Con:Est:God}, a sound and complete Hilbert calculus for each $\mathsf{L}^i_q$ (where $i< q$) can be defined from the one for $\L_{q+1}^{\leq}$ (the degree-preserving counterpart of $\L_{q+1}$) by adding additional inference rules. The negative feature of such approach is that these Hilbert calculi have ``global'' inference rules, that is, inference rules such that one of its permises need to be a theorem of $\L_{q+1}$.
By a general result by Blok and Pigozzi (see Theorem 4.3 in \cite{blok:pig:01}) and from Theorem~\ref{equivsys} in Section~\ref{Joan} below, a standard Hilbert calculus for $\mathsf{L}^i_q$ (for $i< q$) can be obtained from the usual one for $\mathsf{L}^q_q=\L_{q+1}$ by means of translations. That is, such calculi have no ``global'' inference rules. The negative side of this approach is that the resulting axiomatization is obtained by translating connectives from the other logic, and so the resulting calculus can appear as very artificial.
As an alternative, it seems that a direct method for defining a sound and complete ``more natural'' Hilbert calculus for each $\mathsf{L}^i_q$ over a suitable signature can also be  obtained by means of a `separation' technique for the truth-values, similar to the one used in Subection~\ref{sectJ4} to define an alternative axiomatization for ${\sf L}^1_3$. To verify this conjecture is left as an open problem.
Anyway, from now on it will be assumed the existence of a standard Hilbert calculus  ${\sf H}_q^i$ which is sound and complete for the logic $\mathsf{L}^i_q$, where $i< q$.  Of course ${\sf H}_q^q$ will stand for the usual axiomatization of $\L_{q+1}$.
\end{remark}

According to the notation introduced at the beginning of Section \ref{Sect-max}, $i\alpha$ is an abbreviation for the formula $\alpha \oplus\dots\oplus \alpha$ ($i$-times), and  the consequence relation of $\mathsf{L}^i_n$ is denoted by $\vDash_{\mathsf{L}^i_n}$. Recall the following basic property of $\algLfin_{n+1}$.

\begin{lemma} \label{exp-i-prop}
For every $1 \leq i \leq n$ and  $x \in \algLfin_{n+1}$: $ix <  i/n$ iff $x=0$. Thus, $e(i \alpha)< i/n$ iff $e(\alpha)=0$ for every evaluation $e$ in  $\algLfin_{n+1}$ and every formula $\alpha$.
\end{lemma}

From now on, $\bot$ will denote any formula of the form $\neg(p \to p)$, for a propositional variable $p$. Observe that $e(\bot)=0$ for every evaluation in $\algLfin_{n+1}$, every $n \geq 1$ and every propositional variable $p$. This is why the choice of $p$ is inessential for  a concrete construction of $\bot$.

Consider, for $1 \leq i \leq q$, the $i$-explosion law
$$(exp_i) \ \displaystyle \frac{i(\varphi \land \neg\varphi)}{\bot} \ . $$
It is not hard to prove that this rule is not derivable in  any ${\sf H}_q^i$, the sound and complete Hilbert calculus given for the  logic $\mathsf{L}^i_q$ (see Remark \ref{axiomHqi}).

\begin{corollary} \label{exp-i-prop-notder}
For every $1 \leq i \leq q$, the rule $(exp_i)$ is not derivable in ${\sf H}_q^i$.
\end{corollary}
\begin{proof}
We first observe that  if  $p, p'$ are two different  propositional variables, then $i(p \wedge \neg p) \not\vDash_{\mathsf{L}^i_n} p'$ for $1 \leq i \leq n$. Indeed, let $e$ be an evaluation in  $\algLfin_{n+1}$ such that $e(p)\notin\{0,1\}$ and $e(p')=0$. Since  $e(p \wedge \neg p)\neq 0$ then  $e(i (p \wedge \neg p))\geq i/n$, by Lemma \ref{exp-i-prop}. Hence, $i(p \wedge \neg p) \not\vDash_{\mathsf{L}^i_n} p'$. Finally, the corollary then follows  from soundness and completeness of ${\sf H}_q^i$ w.r.t.\ $\mathsf{L}^i_q$.
\end{proof}

However, the $i$-explosion rule is clearly admissible in ${\sf H}^i_q$ since it is a {\em passive} rule, that is: for no instance of the $(exp_i)$ rule, the premisse can be a theorem of ${\sf H}^i_q$. Indeed, for any classical evaluation over $\{0,1\}$ it is the case that $e(\varphi) \in \{0,1\}$ for every formula $\varphi$ and so $e(i (\varphi \wedge \neg \varphi)) < i/n$, by Lemma \ref{exp-i-prop}. This leads us to consider the following definition.

\begin{definition} \label{Hbar}
 $\overline{\sf H}^i_q$ is  the Hilbert calculus obtained from ${\sf H}^i_q$ by adding the $i$-explosion rule $(exp_i)$.  We will denote by $\vdash_{{\sf H}^i_q}$ and $\vdash_{\overline{\sf H}^i_q}$ the notions of proof associated to the Hilbert calculi ${\sf H}^i_q$ and  $\overline{\sf H}^i_q$, respectively.
\end{definition}

The following is a characterization of $\vdash_{\overline{\sf H}^i_q}$ in terms of $\vdash_{{\sf H}^i_q}$.

\begin{proposition} \label{proof-Hqi}
Let $\Gamma \cup\{\varphi\}$ be a set of formulas. Then $\Gamma \vdash_{\overline{\sf H}^i_q} \varphi$ iff either $\Gamma \vdash_{{\sf H}^i_q} \varphi$, or  $\Gamma\vdash_{{\sf H}^i_q} i (\psi \wedge \neg \psi)$ for some formula $\psi$.
\end{proposition}
\begin{proof}
`Only if' part: Suppose that $\Gamma \vdash_{\overline{\sf H}^i_q} \varphi$ such that $\Gamma \nvdash_{{\sf H}^i_q} \varphi$. Then, any derivation in  $\overline{\sf H}^i_q$ of $\varphi$ from $\Gamma$ must use the rule $(exp_i)$. Let $\varphi_1 \ldots \varphi_n$ be a derivation  in  $\overline{\sf H}^i_q$ of $\varphi$ from $\Gamma$. Thus, there exists $1 \leq m < n$ such that $\varphi_m=i (\psi \wedge \neg \psi)$  for some formula $\psi$,  allowing so the first application of $(exp_i)$ in the given derivation. This means that  $\Gamma\vdash_{{\sf H}^i_q} i (\psi \wedge \neg \psi)$, since it was assumed that $(exp_i)$ was not applied before $\varphi_m$ in the given derivation.

`If' part: Suppose that   $\Gamma \vdash_{{\sf H}^i_q} \varphi$. Then, clearly $\Gamma \vdash_{\overline{\sf H}^i_q} \varphi$.
Now, suppose that
$\Gamma\vdash_{{\sf H}^i_q} i (\psi \wedge \neg \psi)$ for some formula $\psi$. Then $\Gamma \vdash_{\overline{\sf H}^i_q} \bot$, by using $(exp_i)$. But $\bot \vDash_{\mathsf{L}^i_q} \varphi$ and so $\bot \vdash_{{\sf H}^i_q} \varphi$, by completeness of ${\sf H}^i_q$ w.r.t.\ $\mathsf{L}^i_q$. This means that $\Gamma \vdash_{\overline{\sf H}^i_q} \varphi$.
\end{proof}

The following question is how to characterize semantically the logic $\overline{\sf H}^i_q$ with respect to  $\mathsf{L}^i_q$, the original sematics for  ${\sf H}^i_q$. The answer  will be obtained in the next section by algebraic arguments (Theorem \ref{final-simple} and Remark \ref{rem-simple}). Indeed, it will be shown there that $\overline{\sf H}^i_q$ is sound and complete w.r.t.\ $\bar{\mathsf L}^i_q$ where, for every $i$ and $n$ with $1 \leq i \leq n$,
$$\bar{\mathsf L}^i_n=\langle \algLfin_{n+1} \times \algLfin_{2}, F_{i/n}\times\{1\}  \rangle$$
such that $\algLfin_{2}$ is the  two-element Boolean algebra   ${\bf B}_2$ with domain $\{0,1\}$.

\section{Translations, equivalent logics and strong maximality} \label{Joan}

\subsection{Preliminaries}

Blok and Pigozzi introduce the notion of equivalent deductive systems in \cite{blok:pig:01} (see also \cite{Blok-Pigozzi:DeductionTheorems}). Two propositional deductive systems $S_{1}$ and $S_{2}$ in the same language $\mathcal{L}$ are equivalent iff there are two translations $\tau_{1}, \tau_{2}$ (finite subsets of $\mathcal{L}$-propositional formulas in one variable) such that:
 \begin{itemize}
 \item $\Gamma\vdash_{S_{1}}\varphi$ iff $\tau_{1}(\Gamma)\vdash_{S_{2}}\tau_{1}(\varphi)$,
 \item $\Delta\vdash_{S_{2}}\psi$ iff $\tau_{2}(\Delta)\vdash_{S_{1}}\tau_{2}(\psi)$,
 \item $\varphi\dashv\vdash_{S_{1}}\tau_{2}(\tau_{1}(\varphi))$,
 \item $\psi\dashv\vdash_{S_{2}}\tau_{1}(\tau_{2}(\psi))$.
\end{itemize}
From very general  results stated in \cite{blok:pig:01} it follows that two equivalent logic systems   are indistinguishable from the point of view of algebra, provided that one of them is algebraizable. Indeed, in such case if one of the systems is algebraizable  then the other will be also algebraizable w.r.t. the same quasivariety. By applying this fact to the systems of the form $\mathsf{L}^{i}_{n}$ studied in the previous sections, several results on relative maximality between these systems and classical logic will be obtained in the next Subsection \ref{CPL}. Actually, these results will be generalized in Subsection \ref{general} to obtain relative maximality results among the systems $\mathsf{L}^{i}_{n}$.   However, for the sake of self containment, we prefer to leave the results of Subsection \ref{CPL} with their simpler proofs as well.

In the rest of this subsection, we provide the necessary preliminaries that will be needed in the subsequent subsections.

We recall that $\L_{\infty}$ is algebraizable and the class $MV$ of all MV-algebras is its equivalent quasivariety semantics \cite{RTV90,CiMuOt99}. Since algebraizability is preserved by finitary extensions then every finite valued $\L$ukasiewicz logic is also algebraizable.

Now we can prove that the deductive systems $\mathsf{L}^{i}_{n}$ and $\mathsf{L}^{j}_{n}$ are in fact equivalent in the above sense. First, observe that, by the McNaughton  functional representation theorem \cite{McNaughton:FunctionalRep}, for every $n\geq2$ and every $1\leq m\leq n$ there is an MV-term $\lambda_{m,n}(p)$ such that for every $a\in [0,1]$,
$$\lambda_{m,n}(a)=\left\{
                    \begin{array}{ll}
                      0, & \hbox{if $a\leq \frac{m-1}{n}$;} \\
                      na-(m-1), & \hbox{if $\frac{m-1}{n}<a<\frac{m}{n}$;} \\
                      1, & \hbox{if $\frac{m}{n}\leq a$.}
                    \end{array}
                  \right.
$$

\begin{lemma}  \label{negFi} The restrictions of the $\lambda_{i,m}$ and $\lambda_{n,n}$ functions on ${\bf LV}_{n+1}$ are the characteristic functions of the order filters $F_{i/n}$ and $F_1$ respectively, i.e.\ for each $a \in {\bf LV}_{n+1}$,
$$\lambda_{i,n}(a)=\left\{
                    \begin{array}{ll}
                      0, & \hbox{if $a < \frac{i}{n}$} \\
                      1, & \hbox{if $ a \geq \frac{i}{n}$}
                    \end{array}
                  \right.
              \qquad
 \lambda_{n,n}(a)= a^n = \left\{
                    \begin{array}{ll}
                      0, & \hbox{if $a < 1$} \\
                      1, & \hbox{if $ a = 1$}
                    \end{array}
                  \right.
$$

\end{lemma}

\begin{theorem} \label{equivsys}
For every $n\geq 2$ and every $1\leq i, j\leq n$, $\mathsf{L}^{i}_{n}$ and $\mathsf{L}^{j}_{n}$ are equivalent deductive systems.
\end{theorem}
\begin{proof}
It is enough to prove that for every $1\leq i\leq n-1$, $\mathsf{L}^{i}_{n}$ is equivalent to $\mathsf{L}^{n}_{n}=\L_{n+1}$.
Let the translations $\tau$ and $\sigma$  be given by $\tau=\{\lambda_{i,n}(p)\}$ and $\sigma=\{\lambda_{n,n}(p)\}$. It is easy to check that for every set of formulas $\Gamma\cup\{\varphi\}$,
$$\Gamma \vDash_{\mathsf{L}^{i}_{n}}\varphi \mbox{ iff  } \{\tau(\psi) \ : \ \psi\in\Gamma\}\vDash_{\mathsf{L}^{n}_{n}}\tau(\varphi)$$
$$\Gamma \vDash_{\mathsf{L}^{n}_{n}}\varphi \mbox{ iff  } \{\sigma(\psi) \ : \ \psi\in\Gamma\}\vDash_{\mathsf{L}^{i}_{n}}\sigma(\varphi)$$
$$\varphi\Dashv\vDash_{\mathsf{L}^{i}_{n}}\sigma(\tau(\varphi)) \mbox{ and } \varphi\Dashv\vDash_{\mathsf{L}^{n}_{n}}\tau(\sigma(\varphi)).$$
Thus, $\mathsf{L}^{i}_{n}$ and $\mathsf{L}^{n}_{n}=\L_{n+1}$ are equivalent deductive systems.
\end{proof}

From the equivalence among $\mathsf{L}^{i}_{n}$ and $\L_{n+1}$, we can obtain, by translating the axiomatization of the finite valued $\L$ukasiewicz logic \L$_{n+1}$,  a calculus sound and complete with respect $\mathsf{L}^{i}_{n}$ that we denote by ${\sf H}^{i}_{n}$ (see \cite[Theorem 4.3]{blok:pig:01}).

Since $\L_{\infty}$ is algebraizable and the class $MV$ of all MV-algebras is its equivalent quasivariety semantics, finitary extensions of $\L_{\infty}$ are in $1$ to $1$ correspondence  with quasivarieties of MV-algebras . Actually, there is a dual isomorphism from the lattice of all finitary extensions of $\L_{\infty}$ and the lattice of all quasivarieties of $MV$. Moreover, if we restrict this correspondence to varieties of MV we get the dual isomorphism from the lattice of all varieties of MV and the lattice of all axiomatic extensions of $\L_{\infty}$. Since $\L_{n+1}=\mathsf{L}^{n}_{n}$ is an axiomatic extension of $\L_{\infty}$,  $\L_{n+1}$  is an algebraizable logic with the class $MV_{n} = \mathcal{Q}({\bf {\L}V}_{n+1})$, the quasivariety generated by ${\bf {\L}V}_{n+1}$, as its equivalent variety semantics. It follows from the previous theorem and from~\cite{blok:pig:01} that $\mathsf{L}^{i}_{n}$, for every $1\leq i\leq n$, is also algebraizable with the same class of $MV_{n}$-algebras as its equivalent variety semantics. Thus, the  lattices of all finitary extensions of $\mathsf{L}^{i}_{n}$ are isomorphic, and in fact, dually isomorphic to the lattice of all subquasivarieties of $MV_{n}$, for all $0<i<n$.

Therefore maximality conditions in the lattice of finitary (axiomatic) extensions  correspond to minimality conditions in the lattice of subquasivarieties (subvarieties). Thus, given two finitary extensions $L_{1}$ and $L_{2}$ of a given logic $\mathsf{L}^{i}_{n}$, where $K_{L_{1}}$ and $K_{L_{2}}$ are its associated $MV_{n}$-quasivarieties, $L_{1}$ is strongly maximal with respect $L_{2}$ iff $K_{L_{1}}$ is a minimal subquasivariety of $MV_{n}$ among those $MV_{n}$-quasivarieties properly containing $K_{L_{2}}$. Moreover, if $L_{1}$ and $L_{2}$ are axiomatic extensions of $\mathsf{L}^{i}_{n}$, then $K_{L_{1}}$ and $K_{L_{2}}$ are indeed $MV_{n}$-varieties. In that case,  $L_{1}$ is  maximal with respect $L_{2}$ iff $K_{L_{1}}$ is a minimal subvariety of $MV_{n}$ among those $MV_{n}$-varieties properly containing $K_{L_{2}}$.

All the axiomatic extensions of $\L_{\infty}$ are characterized by Komori in \cite{Komori:SuperLukasiewiczPropositional}, where it is shown that every axiomatic extension  is finitely axiomatizable and depends only on two finite sets of natural numbers $I, J$ not both empty. Moreover, Panti proved in \cite{Panti:Varieties} that every axiomatic extension can be axiomatized relative to  $\L_{\infty}$ by a single axiom $\gamma_{I,J}$ with a single propositional variable. For the case of finite valued $\L$ukasiewicz logics, Komori's characterization depends on just a finite set of natural numbers in the following sense: given $n>1$, every axiomatic extension of $\L_{n+1}$ is of the form
$$\displaystyle \bigcap_{1\leq j\leq k}\L_{m_{j}+1}$$
for some natural number $k$ where $m_{j}|n$ for every $1\leq j\leq k$. Moreover, from the equivalence of Theorem \ref{equivsys}, it follows that every axiomatic extension of $\mathsf{L}^{i}_{n}$
is of the form
$$\displaystyle \bigcap_{1\leq j\leq k}\mathsf{L}^{i/n}_{m_{j}}$$
for some natural number $k$ where $m_{j}|n$ for every $1\leq j\leq k$,  and it is axiomatized by a single axiom $\gamma^{i/n}_{m_{1},\ldots,m_{k}}$ which depends on one variable.
We denote by ${\sf H}^{i/n}_{m_{1},\ldots,m_{k}}$ the calculus obtained from ${\sf H}^{i}_n$
by adding the axiom $\gamma^{i/n}_{m_{1},\ldots,m_{k}}$. Note that for every $m\geq 1$ such that $m\vert n$, the calculus  ${\sf H}^{i/n}_{m}$ is the same logic as ${\sf H}^{j}_{m}$, where $j$ is the natural number  such that $F_{j/m}=F_{i/n}\cap {\L}V_{m}$.

The lattice of all axiomatic extensions $\L_{\infty}$ is fully described also by Komori in \cite{Komori:SuperLukasiewiczPropositional}, thus from the equivalence of Theorem \ref{equivsys}, we can obtain the following maximality conditions for all axiomatic extensions of ${\sf L}^{i}_{n}$.

\begin{theorem}
Let  $0<i,m\leq  n$ be natural numbers such that $m\vert n$. If $L$ is an axiomatic extension of $\mathsf{L}^i_n$,  then
\begin{itemize}

  \item $L$ is maximal with respect to $\mathsf{L}^{i/n}_{m}$ iff $L=\mathsf{L}^{i/n}_{m}\cap \mathsf{L}^{i/n}_{q^{k+1}}$ for some prime number $q$ with $q | n$ and a natural $k \geq 0$ such that $q^{k}| m$ and $q^{k+1} \nvert m$.

\end{itemize}
\end{theorem}

\begin{proof}
Using the equivalence of Theorem \ref{equivsys} we obtain that the lattice of axiomatic extensions of $\mathsf{L}^{i}_{n}$ is isomorphic to the lattice of axiomatic extensions of $\L_{n+1}$. As mentioned above, every axiomatic extension of $\L_{n+1}$ is characterized by a finite set $\{m_{1}, \ldots, m_{k}\}$ where all
of its elements are divisors of $n$. Given two such sets $\{m_{1}, \ldots, m_{k}\}$ and $\{n_{1}, \ldots, n_{s}\}$, we define the following relation among finite subsets of divisors of $n$: $\{m_{1}, \ldots, m_{k}\}\preceq \{n_{1}, \ldots, n_{s}\}$ iff for every $1\leq i\leq k$ there is $1\leq j\leq s$ such that $m_{i}\vert n_{j}$. This relation $\preceq$ is the dual order of the lattice of axiomatic extensions of $\L_{n+1}$ in the following sense: $\{m_{1}, \ldots, m_{k}\}\preceq \{n_{1}, \ldots, n_{s}\}$  iff $\displaystyle \bigcap_{1\leq j\leq s}\L_{n_{j}+1}\leq \displaystyle \bigcap_{1\leq i\leq k}\L_{m_{i}+1}$. Clearly,  $\{m\}\preceq \{m, q\}$ and $\{m, q\}\not\preceq\{m\}$ if $q$ is a prime number such that $q\vert n$ and $q \nvert m$; Similarly,   $\{m\}\preceq\{m, q^{k+1}\}$ and $\{m, q^{k+1}\}\not\preceq\{m\}$ if $q$ is a prime number such that $q\vert n$,  $q^{k}| m$ and $q^{k+1} \nvert m$. Moreover if $\{m\}\preceq\{m_{1}, \ldots, m_{k}\}$ and $\{m_{1}, \ldots, m_{k}\}\not\preceq\{m\}$, then there is $m_{i}$ such that $m\vert m_{i}$.  If $m\neq m_{i}$ then there is a prime number $q$ such that $mq \vert m_{i}\vert n$. Thus   $\{m, q\}\preceq\{m_{1},\ldots, m_{k}\}$ if $q \nvert m$ and $\{m, q^{k+1}\}\preceq\{m_{1},\ldots, m_{k}\}$ if  $q^{k}| m$ and $q^{k+1} \nvert m$. If $m= m_{i}$, then  there is an $m_{j}$ with $1<j<k$ such that $m_{j}\nvert m$. If there is a prime number $q$ such that $q\vert m_{j}\vert n$ such that $q \nvert m$, then $\{m, q\}\preceq\{m_{1},\ldots, m_{k}\}$. Otherwise, there are a prime number $q$ and a natural $k>0$ such that $q^{k+1}\vert m_{j}\vert n$, $q^{k}| m$ and $q^{k+1} \nvert m$, then $\{m, q^{k+1}\}\preceq\{m_{1},\ldots, m_{k}\}$. Duality and Theorem \ref{equivsys} close the proof.
\end{proof}

As a corollary we obtain that the suficient condition of Proposition \ref{maxLnLm} is also necessary.

\begin{corollary}
Let $1 \leq i,m \leq n$. Then $\mathsf{L}^i_n=\langle \algLfin_{n+1},F_{i/n}\rangle$ is maximal w.r.t.\  $\mathsf{L}_{m}^{i/n}=\langle \algLfin_{m+1}, F_{i/n} \cap \domLfin_{m+1}
\rangle$ if and only if  there is some prime number $q$ and $k \geq 1$ such that $n=q^k$, and  $m=q^{k-1}$.
\end{corollary}

The task of fully describing the lattice of all  all finitary extensions of $\L_{\infty}$, isomorphic to the lattice of all subquasivarieties of $MV$,  turns to be an heroic task since the class of all MV-algebras is $Q$-universal (see~\cite{AdDz}). For the finite valued case it is much simpler, since $MV_{n}$ is a locally finite discriminator variety (cf.~\cite{B_C_V, GT14}). Any locally finite quasivariety is generated by its critical algebras (see~\cite{Dz0}), where an algebra $A$ is said to be \emph{critical} iff it is a finite algebra not belonging to the quasivariety generated by all its proper subalgebras. A  description of all critical MV-algebras can be found in~\cite{GT14}.

\begin{theorem}\emph{\cite[Theorem 2.5]{GT14}}
\label{Wcritical}  An MV-algebra $A$ is critical if and only if $A$ is isomorphic to a finite MV-algebra $\mathbf{LV}_{n_{0}+1}\times\cdots\times\mathbf{LV}_{n_{l-1}+1} $ satisfying the following conditions:
\begin{enumerate}
 \item For every $i,j<l$, $ i\neq j$ implies $ n_{i}\neq n_{j}$.
 \item If there exists $n_{j}$, $j<l$ such that $n_{i}|n_{j}$ for some
  $i\neq j$, then $n_{j}$ is unique.
\end{enumerate}
\end{theorem}

Moreover the following result characterizes the inclusion among locally finite quasivarieties.

\begin{lemma} \emph{\cite[Lemma 2.9]{GT14}} \label{previous}
\label{distinct}
Let $\mathfrak{F} = \{\mathbf{ LV}_{n_{i1}+1}\times\cdots\times\mathbf{ LV}_{n_{il(i)}+1} \ : \ i\in I\}$ and
$\mathfrak{G} = \{\mathbf{ LV}_{m_{j1}+1}\times\cdots\times\mathbf{ LV}_{m_{jl(j)}+1} \ : \ j\in J\}$ be two finite families of
critical MV-algebras. Then it holds that
$$ \mathcal{Q}(\mathfrak{F}) \subseteq  \mathcal{Q}(\mathfrak{G})$$
if, and only if,
for every $i\in I$ there exists  a non-empty $H \subseteq J$ such that:
\begin{enumerate}
 \item For any $1\le k\le l(i)$ there are $j\in H$ and $1\le r\le l(j)$ such that
   $n_{ik}|m_{jr}$.
 \item For any $j\in H$ and $1\le r\le l(j)$ there exists $1\le k\le l(i)$ such that
   $n_{ik}|m_{jr}$.
\end{enumerate}
\end{lemma}

\subsection{Strong maximality among logics $\mathsf{L}^i_q$, $\bar{\sf L}^i_q$, and classical logic}
\label{CPL}

As a direct application of Lemma \ref{distinct}, we have the following particular case that will be used later.

\begin{corollary} \label{corollary-simple}
Consider the  following two sets of one critical MV-algebra each: \\
$\{\mathbf{ LV}_{q+1}\times \mathbf{ LV}_{2}\}$ and
$\{\mathbf{ LV}_{k+1}\}$, where $q$ is a prime number such that  $q > 1$. Then
\[\mathcal{Q}(\{\mathbf{ LV}_{q+1}\times \mathbf{ LV}_{2}\}) \subseteq \mathcal{Q}(\mathbf{ LV}_{k+1})\]
if and only if  $q | k$.
\end{corollary}

\begin{proof}  The two families of critical algebras above correspond in Lemma \ref{distinct} to take  $I = \{1\}$ and $J = \{1\}$, with $n_{11} = q$, $n_{12} = 1$, $m_{11} = k$. Then one can check that these values satisfy the two conditions of the lemma only in the case that $q | k$.
\end{proof}

Now, for any $k > 1$, we are able to provide a full description of the minimal subquasivarieties of $MV_{k}=\mathcal{Q}(\mathbf{LV}_{k+1})$ strictly containing the variety of Boolean algebras.

\begin{theorem} \label{minimal-simple}  
Let $k>1$.
The set of all minimal subquasivarieties of $MV_{k}=\mathcal{Q}(\mathbf{LV}_{k+1})$ among those strictly containing the class of all the Boolean algebras ${\bf B} = \mathcal{Q}(\mathbf{LV}_{2})$ is $$M^{k}=\{\mathcal{Q}(\mathbf{LV}_{q+1}\times\mathbf{LV}_{2}) \ : \ q > 1 \mbox{ prime, } q|k  \}.$$
\end{theorem}

\begin{proof}
 By Lemma \ref{previous} and the previous Corollary \ref{corollary-simple}, every $K\in M^{k}$ is a subquasivariety of $\mathcal{Q}(\mathbf{LV}_{k+1})$ strictly containing ${\bf B}$. Moreover, for every $K_{1}, K_{2}\in M^{k}$, if $K_{1}\neq K_{2}$ then $K_{1}\not\subseteq K_{2}$ and $K_{2}\not\subseteq K_{1}$.

On the other hand, let $K$ be a minimal subquasivariety of $\mathcal{Q}(\mathbf{LV}_{k+1})$ strictly containing ${\bf B}$.
Since $K\neq {\bf B}$, it must contain a critical algebra $C$ that, by Theorem \ref{Wcritical}, it must be such that $C\cong\mathbf{ LV}_{m_{1}+1}\times\cdots\times\mathbf{ LV}_{m_{s}+1}$, where   $m_{i}| k$ for every $1 \leq i \leq s$, and $m_{j} > 1$ for some $1\leq j\leq s$.
Hence, for every prime number $q$ such that $q|m_{j}$, and hence $q | k$, we have $\mathbf{LV}_{q+1}\times\mathbf{LV}_{2}\in \mathcal{Q}(C)\subseteq K$, and thus  $\mathcal{Q}(\mathbf{LV}_{q+1}\times\mathbf{LV}_{2}) \subseteq K$. Since we are assuming the minimality of $K$, it must be $\mathcal{Q}(\mathbf{LV}_{q+1}\times\mathbf{LV}_{2}) = K$.
\end{proof}

\begin{theorem}\label{axiomminimal-simple}
If $q > 0$ is a prime number, then $\mathcal{Q}(\mathbf{LV}_{q+1}\times\mathbf{LV}_{2})$ is axiomatized by the MV quasi-identities plus:
\begin{itemize}
\item $\gamma_{q}(x)\approx 1$ (the identity axiomatizing $\mathcal{V}(\mathbf{LV}_{q+1})$)
\item $q(x\land\lnot x)\approx 1 \Rightarrow y \lor \neg y \approx 1$
\end{itemize}
\end{theorem}
\begin{proof}
It is easy to check that $\mathbf{LV}_{q+1}\times\mathbf{LV}_{2}$ satisfies these two quasi-identities.
Since the MV-identities and  $\gamma_{q}(x)\approx 1$ axiomatize $\mathcal{V}(\mathbf{LV}_{q+1})$, and $\mathcal{V}(\mathbf{LV}_{q+1})$ is a locally finite quasivariety, it is enough to prove that every critical MV-algebra $C\in\mathcal{V}(\mathbf{LV}_{q+1})$ where the quasi-equation $q(x\land\lnot x)\approx 1 \Rightarrow y \lor \neg y \approx 1$ holds, belongs to $\mathcal{Q}(\mathbf{LV}_{q+1}\times\mathbf{LV}_{2})$.

Let $C$ be a critical MV-algebra satisfying the axiomatization, then $C$ is such that
$C\cong\mathbf{ LV}_{m_{1}+1}\times\cdots\times\mathbf{ LV}_{m_{k}+1}$ satisfying conditions of Theorem \ref{Wcritical}. Moreover, every $1 \leq i \leq  k$, is such that either $m_{i} = 1$ or $m_{i} = q$ because  $\mathbf{ LV}_{m_{i}+1}$ belongs to $\mathcal{V}(\mathbf{LV}_{q+1})$. If there is $c\in C$ such that $q(c\land\lnot c)=1$ then, by the second quasi-equation of the above axiomatization,
$b \lor \neg b \approx 1$ for any $b\in C$. Thus we have $C\in {\bf B} \subseteq
\mathcal{Q}(\mathbf{LV}_{q+1}\times\mathbf{LV}_{2})$. Otherwise, recalling that either $m_{i} = 1$ or $m_{i} = q$ for every $i$, if for every $c\in C$ one has $q(c\land \lnot c)\neq 1$ then  $m_{i}=1$ for some $1\leq i\leq k$. In that case, by the characterization of critical algebras (Theorem \ref{Wcritical}), we have $C\cong \mathbf{LV}_{2}$ or $C\cong \mathbf{LV}_{q+1}\times\mathbf{LV}_{2}$. If $C\cong \mathbf{LV}_{2}$, then trivially  $C\in
\mathcal{Q}(\mathbf{LV}_{q+1}\times\mathbf{LV}_{2})$. If $C\cong \mathbf{LV}_{q+1}\times\mathbf{LV}_{2}$,
then clearly $C \in\mathcal{Q}(\mathbf{LV}_{q+1}\times\mathbf{LV}_{2})$.
\end{proof}
Above, note that the identity $y \lor \neg y \approx 1$ corresponds to the previously mentioned Panti's axiom $\gamma_{I, J}(y)$, with $I = \{1\}$ and $J = \emptyset$,  axiomatizing {\sf CPL} as an axiomatic extension of $\L_{n+1}$ for any $n > 1$.

Finally, we obtain the following characterization result about strong maximality of logics  $\bar{\sf L}^{j}_{q}$ with respect to classical logic.

\begin{theorem}  \label{final-simple}
Let  $q > 1$ be a prime number. Then, for every $j$ such that $0 < j \leq q$:
\begin{itemize}
 \item
 $ \bar{\sf L}^{j}_{q}$ is  strongly maximal with respect to {\sf CPL}  and it is axiomatized by 
  ${\sf H}^j_q$ plus the rule  $j(\varphi\land\lnot \varphi)/ (\psi \lor \neg \psi)^q$. 

  \item  $\mathsf{L}^{j}_{q}$ is strongly maximal w.r.t. $\bar{\sf L}^{j}_{q}$.
\end{itemize}
\end{theorem}

\begin{proof}
 By using the equivalence of Theorem \ref{equivsys} and the algebraizability of $\L_{q+1}$, the lattice of subquasivarieties of $\mathcal{V}(\mathbf{LV}_{q+1})$ is dually order isomorphic to the lattice of all finitary extensions of $\L_{q+1}$. Clearly ${\sf CPL} = \L_2$ is the finitary extension of $ \bar{\sf L}^{j}_{q}$ corresponding to the subvariety $\mathcal{Q}(\mathbf{LV}_{2})$ of  $\mathcal{Q}(\mathbf{LV}_{q+1}\times\mathbf{LV}_{2})$, and
$ \bar{\sf L}^{j}_{q}$
is the finitary extension of $\mathsf{L}^{j}_{q}$ corresponding to the subquasivariety  $\mathcal{Q}(\mathbf{LV}_{q+1}\times\mathbf{LV}_{2})$ of $\mathcal{V}(\mathbf{LV}_{q+1})$.

By Theorem \ref{Wcritical}, the only critical algebras of $\mathcal{V}(\mathbf{LV}_{q+1})$ are   $\mathbf{LV}_{q+1}$, $\mathbf{LV}_{2}$ and $\mathbf{LV}_{2} \times \mathbf{LV}_{q+1}$ and, by   Lemma \ref{distinct}, all its  subquasivarieties are $\mathcal{Q}(\mathbf{LV}_{2}) \subsetneq \mathcal{Q}(\mathbf{LV}_{2} \times \mathbf{LV}_{q+1}) \subsetneq \mathcal{Q}(\mathbf{LV}_{q+1})$. Therefore,  by Theorem \ref{minimal-simple}, $ \bar{\sf L}^{j}_{q}$
is strongly maximal with respect to {\sf CPL}, while ${L}^{j}_{q}$
is strongly maximal with respect to $\bar{\sf L}^{j}_{q}$.

Finally, the axiomatization of  $\bar{\sf L}^{j}_{q}$ follows from Theorem \ref{axiomminimal-simple} and the facts that $j\,\varphi\Dashv\vDash_{\mathsf{L}^{j}_{q}}q\ \varphi$  holds for every formula $\varphi$ and that the equation $(q x)^q = q x$ is valid in the class $MV_q$.
\end{proof}

From the above proof, it readily follows the next corollary.

\begin{corollary} \label{between}
 $\bar{\sf L}^{j}_{q}$ is the unique strongly maximal logic w.r.t.\ {\sf CPL} above ${\sf L}^j_{q}$. In fact, $\bar{\sf L}^{j}_{q}$ is the only logic between ${\sf L}^j_{q}$ and {\sf CPL}.
\end{corollary}

\begin{remark} \label{rem-simple}
It is worth noting that the rule $j(\varphi\land\lnot \varphi)/ (\psi \lor \neg \psi)^q$ exactly corresponds to the explosion rule $(exp_j)$ introduced in Section \ref{sectLqi}. Indeed,  the rule $j(\varphi\land\lnot \varphi)/ (\psi \lor \neg \psi)^q$  is clearly derivable from $(exp_j)$. On the other hand, assuming $j(\varphi\land\lnot \varphi)$, by this rule it follows that $(\psi \lor \neg \psi)^q$ for every $\psi$. Hence the logic becomes {\sf CPL} because the translation of the classical axiom $\psi \lor \neg \psi$ is precisely $(\psi \lor \neg \psi)^q$, and thus $\bot$ follows from $j(\varphi\land\lnot \varphi)$.
This does not come as a surprise, since as we have proved above, $\mathsf{L}^{j}_{q}$ is strongly maximal w.r.t. $\bar{\sf L}^{j}_{q}$ and so the latter is the only proper extension of $\mathsf{L}^{j}_{q}$  (with a finitary rule) properly contained in {\sf CPL}.
\end{remark}

As a corollary of the previous remark, it follows the completeness of $\overline{\sf H}^j_q$.

\begin{corollary} \label{cor-after-remark} $\overline{\sf H}^j_q$ is sound and complete w.r.t.\  $\bar{\sf L}^{j}_{q}$.
\end{corollary}

\subsection{Strong maximality with respect to systems $\mathsf{L}^{i}_{n}$}
\label{general}

Next theorems are generalizations of Theorems \ref{minimal-simple}, \ref{axiomminimal-simple} and \ref{final-simple} respectively.

\begin{theorem} \label{minimal}
Let $n>0$ and $k>1$.
The set of all minimal subquasivarieties of $MV_{nk}=\mathcal{Q}(\mathbf{LV}_{nk+1})$ among those strictly containing $\mathcal{Q}(\mathbf{LV}_{n+1})$ is \\

\noindent \begin{tabular}{rl}
$M^{nk}_{n} =$&$\{ \mathcal{Q}(\{\mathbf{LV}_{n+1},\mathbf{LV}_{q+1}\times\mathbf{LV}_{2}\}) \ : \ q \mbox{ prime, } q | k \mbox{ and } q\nvert n \} \bigcup$ \\
& $\{\mathcal{Q}(\{\mathbf{LV}_{n+1}, \mathbf{LV}_{q^{r+1}+1}\times\mathbf{LV}_{2}\}) \ : \ q \mbox{ prime, } q|k,  q^{r}|n \mbox{ and } q^{r+1}\nvert n \}.$ \\
\end{tabular}
\end{theorem}

\begin{proof}
 By Lemma \ref{previous}, every $K\in M^{nk}_{n}$ is a subquasivariety of $\mathcal{Q}(\mathbf{LV}_{nk+1})$ strictly containing $\mathcal{Q}(\mathbf{LV}_{n+1})$. Moreover, for every $K_{1}, K_{2}\in M^{nk}_{n}$, if $K_{1}\neq K_{2}$ then $K_{1}\not\subseteq K_{2}$ and $K_{2}\not\subseteq K_{1}$.

Let $K$ be a minimal subquasivariety of $\mathcal{Q}(\mathbf{LV}_{nk+1})$ strictly containing $\mathcal{Q}(\mathbf{LV}_{n+1})$. Trivially, $\mathbf{LV}_{n+1}\in K$. Since $K\neq \mathcal{Q}(\mathbf{LV}_{n+1})$, then it must contain a critical algebra $C\cong\mathbf{ LV}_{m_{1}+1}\times\cdots\times\mathbf{ LV}_{m_{s}+1}$ such that  $m_{i}|nk$ for every $1 \leq i\leq s$ and $m_{j} \nvert n$ for some $1\leq j\leq s$. If there is a prime number $q|m_{j}$ such that $q\nvert n$, then $\mathbf{LV}_{q+1}\times\mathbf{LV}_{2}\in \mathcal{Q}(C)\subseteq K$. Otherwise, there is a prime $q$ such that  $q|m_{j}$  and $q^{r}|n$, and for some $r\geq 1$, $q^{r+1}\! \not | n$ and $q^{r+1}|m_{j}$, whence $\mathbf{LV}_{q^{r+1}+1}\times\mathbf{LV}_{2}\in \mathcal{Q}(C)\subseteq K$. Thus, in both cases $ K$ contains  some $K_{i}\in M^{nk}_{n}$, from which it follows that $K \in M^{nk}_{n}$ since we are assuming minimality of $K$.
\end{proof}

\begin{theorem}\label{axiomminimal}
For every  $n>0$.
\begin{itemize}
\item If $q$ is a prime number such that $q \nvert n$, 
then $\mathcal{Q}(\{\mathbf{LV}_{n+1},\mathbf{LV}_{q+1}\times\mathbf{LV}_{2}\})$ is axiomatized by the MV identities plus
\begin{itemize}
\item  $\gamma_{\{n, q\},\emptyset}(x)\approx 1 \mbox{ (the identity axiomatizing }
\mathcal{V}(\{\mathbf{LV}_{n+1},\mathbf{LV}_{q+1}\})\mbox{})$
\item  $nq(x\land\lnot x)\approx 1 \Rightarrow \gamma_{\{n\},\emptyset}(y)\approx 1.$
\end{itemize}
 \item If $q$ is a prime number such that $q^{r}| n$ and $q^{r+1} \nvert n$, 
 $\mathcal{Q}(\{\mathbf{LV}_{n+1}, \mathbf{LV}_{q^{r+1}+1}\times\mathbf{LV}_{2}\})$ is axiomatized by the MV identities plus
 \begin{itemize}
 \item $\gamma_{\{n, q^{r+1}\},\emptyset}(x)\approx 1\mbox{ (the identity axiomatizing } \mathcal{V}(\{\mathbf{LV}_{n+1},\mathbf{LV}_{q^{r+1}+1}\}) \mbox{} )$
 \item $nq(x\land\lnot x)\approx 1 \Rightarrow \gamma_{\{n\},\emptyset}(y)\approx 1.$
 \end{itemize}
\end{itemize}
\end{theorem}

\begin{proof} We prove the first item, the other is proved in a analogous way.
It is easy to check that $\mathbf{LV}_{n+1}$ and $\mathbf{LV}_{q+1}\times\mathbf{LV}_{2}$ satisfy all the quasi-identities.
Since the MV-identities with  $\gamma_{\{n, q\},\emptyset}(x)\approx 1$ axiomatize $\mathcal{V}(\{\mathbf{LV}_{n+1},\mathbf{LV}_{q+1}\})$ and $\mathcal{V}(\{\mathbf{LV}_{n+1},\mathbf{LV}_{q+1}\})$ is a locally finite quasivariety, it is enough to prove that every critical MV-algebra $C\in\mathcal{V}(\{\mathbf{LV}_{n+1},\mathbf{LV}_{q+1}\})$ where the quasiequation $nq(x\land\lnot x)\approx 1 \Rightarrow \gamma_{\{n\},\emptyset}(y)\approx 1$ holds, belongs to $\mathcal{Q}(\{\mathbf{LV}_{n+1},\mathbf{LV}_{q+1}\times\mathbf{LV}_{2}\})$. Therefore, let $C$ be a critical MV-algebra satisfying the axiomatization. Then, $C$ is such that
$C\cong\mathbf{ LV}_{m_{1}+1}\times\cdots\times\mathbf{ LV}_{m_{r}+1}$ satisfying conditions of Theorem \ref{Wcritical}, and moreover for every $1 \leq i \leq k$,  either $m_{i}|n$ or $m_{i} = q$ because  $C \in \mathcal{V}(\{\mathbf{LV}_{n+1},\mathbf{LV}_{q+1}\})$. If there is $c\in C$ such that $nq(c\land\lnot c)=1$ then, by the second quasi-equation of the axiomatization above,
$\gamma_{\{n\},\emptyset}(b)\approx 1$ for any $b\in C$, thus $C\in \mathcal{V}(\{\mathbf{LV}_{n+1}\})=\mathcal{Q}(\{\mathbf{LV}_{n+1}\})\subseteq
\mathcal{Q}(\{\mathbf{LV}_{n+1},\mathbf{LV}_{q+1}\times\mathbf{LV}_{2}\})$. If for every $c\in C$, $nq(c\land \lnot c)\neq 1$ then  $m_{i}=1$ for some $1\leq i\leq k$. In that case, by the characterization of critical algebras (Theorem \ref{Wcritical}), either $C\cong \mathbf{LV}_{2}$ or $C\cong \mathbf{LV}_{m+1}\times\mathbf{LV}_{2}$. If $C\cong \mathbf{LV}_{2}$, then trivially  $C\in
\mathcal{Q}(\{\mathbf{LV}_{n+1},\mathbf{LV}_{q+1}\times\mathbf{LV}_{2}\})$. Otherwise, if $C\cong \mathbf{LV}_{m+1}\times\mathbf{LV}_{2}$, since $C\in\mathcal{V}(\{\mathbf{LV}_{n+1},\mathbf{LV}_{q+1}\})$, either $m|n$ or $m=q$. If $m|n$ then $C\in \mathcal{V}(\{\mathbf{LV}_{n+1}\})=\mathcal{Q}(\{\mathbf{LV}_{n+1}\})\subseteq
\mathcal{Q}(\{\mathbf{LV}_{n+1},\mathbf{LV}_{q+1}\times\mathbf{LV}_{2}\})$. If $m=q$ then $C\cong \mathbf{LV}_{q+1}\times\mathbf{LV}_{2}\in\mathcal{Q}(\{\mathbf{LV}_{n+1},\mathbf{LV}_{q+1}\times\mathbf{LV}_{2}\})$.
\end{proof}

If $1\leq i,m\leq n$, by analogy with $\mathsf{L}^{i/n}_m$, we define the matrix logic $$\bar{\sf L}^{i/n}_{m}=\langle \algLfin_{m+1} \times \algLfin_{2}, (F_{i/n}\cap\algLfin_{m+1}) \times\{1\}  \rangle.$$
Then we have the following generalization of Theorem \ref{final-simple}.

\begin{theorem}
Let  $0<i\leq n$ be natural numbers  and let $q$ be a prime number. Then we have:
\begin{itemize}
 \item If
 $q \nvert n$ then, for every $j$ such that $(i-1)q<j\leq iq$, $\mathsf{L}^{i}_{n}\cap \bar{\sf L}^{j/nq}_{q}$ is  strongly maximal with respect to $\mathsf{L}^{i}_{n}$,  and it is axiomatized by ${\sf H}^{j/nq}_{n, q}$ plus the rule $j(\varphi\land\lnot \varphi)/\gamma^{j/nq}_{n}(\psi)$.

  \item If $q^{r}| n$ and $q^{r+1} \nvert n$  then, for every $j$ such that $(i-1)q<j\leq iq$, $\mathsf{L}^{i}_{n}\cap \bar{\sf L}^{j/nq}_{q^{r+1}}$ is  strongly maximal with respect  to $\mathsf{L}^{i}_{n}$,  and it is axiomatized by ${\sf H}^{j/nq}_{n, q^{r+1}}$ plus the rule  $j(\varphi\land\lnot \varphi)/\gamma^{j/nq}_{n}(\psi)$.

\end{itemize}
Recall that  in the above rules $\gamma^{j/nq}_{n}(\psi)$ refers to the axiom in one variable that axiomatizes  $\mathsf{L}^{j/nq}_{n}$ relative to $\mathsf{L}^{j}_{nq}$.
Moreover, every finitary  extension of some $\mathsf{L}^{j}_{k}$ is strongly maximal with respect $\mathsf{L}^{i}_{n}$ iff it is of one of the two preceeding types.
\end{theorem}

\begin{proof}
Notice that $\mathsf{L}^{i}_{n}=\mathsf{L}^{j/nq}_{n}$ for  every $j$ such that $(i-1)q<j\leq iq$.  Thus,  $\mathsf{L}^{i}_{n}$ is an extension of $\mathsf{L}^{j}_{nq}$. Now, by using the equivalence of Theorem \ref{equivsys} and the algebraizability of $\L_{nq}+1$, the lattice of subquasivarieties of $\mathcal{V}(\mathbf{LV}_{nq+1})$ is dually order isomorphic to the lattice of all the finitary extensions of $\mathsf{L}^{j}_{nq}$. Moreover, $\mathsf{L}^{i}_{n}\cap \bar{\sf L}^{j/nq}_{q}$  and  $\mathsf{L}^{i}_{n}\cap \bar{\sf L}^{j/nq}_{q^{r+1}}$ are the finitary extensions of $\mathsf{L}^{j}_{nq}$ associated to $\mathcal{Q}(\{\mathbf{LV}_{n+1},\mathbf{LV}_{q+1}\times\mathbf{LV}_{2}\})$ and $\mathcal{Q}(\{\mathbf{LV}_{n+1}, \mathbf{LV}_{q^{r+1}+1}\times\mathbf{LV}_{2}\})$, respectively.  Hence, they are strongly maximal with respect to $\mathsf{L}^{i}_{n}$, by Theorem \ref{minimal}.  The axiomatization follows from Theorem \ref{axiomminimal} and the facts  that $j\,\varphi\Dashv\vDash_{\mathsf{L}^{j}_{nq}}nq\ \varphi$
holds for every formula $\varphi$ and that the equation $(nq x)^{nq} = nq x$ is valid in the class $MV_{nq}$.

Finally, the last statement of this theorem follows from Theorem \ref{minimal} and Theorem \ref{equivsys}.
\end{proof}

%
%

\section{An application to ideal paraconsistent logics}  \label{sectIdeal}

As mentioned in Example \ref{EjIdeal}, Arieli et al.\ 
have introduced in  \cite{ArieliAZ11a}  the concept of {\em ideal paraconsistent logics}. We recall here this notion.

\begin{definition} [c.f. \cite{ArieliAZ11a}] \label{IdPar}
Let $L$ be a propositional logic defined over a signature $\Theta$ (with consecuence relation $\vdash_L$) containing at least a unary connective $\neg$ and a binary connective $\to$ such that:
\begin{itemize}
\item[(i)] $L$ is paraconsistent w.r.t.\ $\neg$ (or simply $\neg$-paraconsistent), that is,  there are formulas $\varphi,\psi \in \mathcal{L}(\Theta)$ such that $ \varphi, \neg\varphi \nvdash_L\psi$;
\item[(ii)] $\to$ is an implication for which the deduction-detachment theorem holds in $L$, that is, $\Gamma \cup \{\varphi\} \vdash_L \psi$ iff  $\Gamma \vdash_L \varphi \to\psi$, for every set for formulas $\Gamma \cup \{\varphi, \psi\} \subseteq  \mathcal{L}(\Theta)$.
\item[(ii)] There is a presentation of {\sf CPL} as a matrix logic  $L'=\langle \mathbf{A}, \{1\}\rangle$  over the signature $\Theta$ such that the domain of $\mathbf{A}$ is $ \{0,1\}$, and $\neg$ and $\to$ are interpreted as the usual 2-valued negation and implication of {\sf CPL}, respectively.

\item[(iv)] $L$ is a sublogic of {\sf CPL} in the sense that $\vdash_L \subseteq \; \vdash_{L'}$, that is, $\Gamma \vdash_L \varphi$ implies  $\Gamma \vdash_{L'} \varphi$,  for every set for formulas $\Gamma \cup \{\varphi\} \subseteq  \mathcal{L}(\Theta)$.
\end{itemize}
Then, $L$ is said to be an {\em ideal paraconsistent logic} if it is maximal w.r.t.\ $L'$, and every proper extension of $L$ over $\Theta$ is not $\neg$-paraconsistent.
\end{definition}

An implication connective satisfying the above condition (ii) will be called {\em deductive implication} in the rest of the paper.\footnote{Such an implication is called {\em deductive} in \cite{car:con:mar:07,CF2014} and {\em proper} in \cite{ArieliAZ11a}. }

\

Thus, a $\neg$-paraconsistent logic $L$ with a deductive implication is ideal if it is maximal w.r.t. {\sf CPL} (presented over the signature $\Theta$ of $L$) and, if $L''$ is another logic over $\Theta$ properly containing $L$, with $\Gamma \cup \{\varphi\} \subseteq  \mathcal{L}(\Theta)$ such that $\Gamma \vdash_{L''}\varphi$ but  $\Gamma \nvdash_L \varphi$, then the logic obtained from $L$ by adding $\Gamma / \varphi$ as an inference rule is not $\neg$-paraconsistent.

As already noticed, the logics ${\sf L}_i^n$ with $i/n \leq 1/2$ are paraconsistent. In this section, using the results of the previous sections, we study the status of the logics ${\sf L}^i_n$ in relation to ideal paraconsistency. Namely,
in the following subsection, we will show that the logics of the form ${\sf L}_q^i$, where $q$ is prime and $i/q \leq 1/2$ are ideal paraconsistent,
while in subsection \ref{sectJ4} the special case of  ${\sf L}_3^1$, renamed as ${\sf J}_4$, is analyzed in more detail.

\subsection{The ideal paraconsistent logics ${\sf L}_q^i$} \label{sectLiq}

By combining Proposition~\ref{parLqi} with  Corollary~\ref{maxLqi} we know a logic $\mathsf{L}^i_q$ is $\neg$-paracon\-sis\-tent and  maximal w.r.t.\ {\sf CPL}, provided that $q$ is prime and $i/q \leq 1/2$. From now on we will assume this is the case when referring to a logic $\mathsf{L}^i_q$.

Recall 
that
 $\overline{\sf H}^i_q$ is  the Hilbert calculus obtained from the calculus ${\sf H}^i_q$ for  $\mathsf{L}^i_q$ by adding the $i$-explosion rule $(exp_i)$. Since $\varphi \wedge \neg\varphi \vdash_{{\sf H}^i_q} i(\varphi \wedge \neg\varphi)$, the logic  $\overline{\sf H}^i_q$  is explosive. Then, taking into account Corollary \ref{between},
 it follows that  every proper extension of $\mathsf{L}^i_q$ defined over its signature is either $\bar{\mathsf{L}}^i_q$ or {\sf CPL}, and hence not $\neg$-paraconsistent.

In addition, by Lemma \ref{negFi}, we know there is a definable unary connective ${\sim}^i_q$ such that, for every evaluation $e$,
$e({\sim}^i_q~p)=0$ if $e(p)  \geq i/q$, and  $e({\sim}^i_q~p)=1$ otherwise, for every propositional variable $p$.\footnote{Namely, ${\sim}^i_q~p = \neg\lambda_{i,q}(p)$.} This is a kind of ``classical'' negation defined on ${\sf L}_q^i$. Using this negation, one can define in turn a new implication $\Rightarrow^i_q$ by stipulating $\varphi \Rightarrow^i_q \psi={\sim}^i_q \varphi \lor \psi$. In fact, one can easily check that $\Rightarrow^i_q$ is a deductive implication on  ${\sf L}^i_q$ in the sense of Definition \ref{IdPar} and that
over $\{0,1\}$ it coincides with the classical implication.
All the above considerations lead to the following result.

\begin{proposition} \label{IdealLiq}
Let  $q$ is a prime number, and let $1 \leq i < q$ such that $i/q \leq 1/2$. Then,  $\mathsf{L}^i_q$ is a $(q+1)$-valued ideal paraconsistent logic.\footnote{Strictly speaking, in this claim we implicitly assume that the signature of $\mathsf{L}^i_q$ has been changed by adding the definable implication $\Rightarrow^i_q$ as a primitive connective. }
\end{proposition}

Therefore we have a large family of examples of ideal paraconsistent logics. In particular, for each prime $q$, all the logics in the set $PC_{q+1} = \{ \mathsf{L}^i_q \ : \  i < q/2 \}$ are $(q+1)$-valued ideal paraconsistent logics. Moreover, if we consider ``the more theorems a paraconsistent logic has, the more well-behaved is the logic'' as a valid further criterion, then we can still refine the set $PC_{q+1}$. Indeed, if we denote by $Th(L)$ the set of theorems of a logic $L$ then, as noticed in Remark~\ref{Lqi-indist}, we have the strict inclusions $Th(\mathsf{L}^i_q) \subsetneq Th(\mathsf{L}^j_q) \subsetneq Th({\sf CPL})$ whenever $i > j$. Therefore the logic ${\sf J}_{q+1} = \mathsf{L}^1_q$ appears to be the ``best'' ideal logic in the set $PC_{q+1}$,\footnote{We have chosen the name ${\sf J}_{q+1}$ to denote the logic $\mathsf{L}^1_q$ inspired in the 3-valued case, where the ideal paraconsistent logic ${\sf J}_3$ coincides with $\mathsf{L}^1_2$.} since it is the logic in that set having the biggest set of theorems from classical logic.

Finally, it is worth mentioning that all the paraconsistent logics of the form $\mathsf{L}^i_n$ are, indeed, {\bf LFI}s (recall Section~\ref{recovery}):

\begin{proposition} \label{LFI-Lqi} Suppose that $i/n \leq 1/2$. Then, the logic $\mathsf{L}^i_n$ is an {\bf LFI} w.r.t.\ $\neg$ and where the consistency operator is defined as $\circ \alpha = {\sim}^i_n (\alpha \wedge \neg \alpha)$.
\end{proposition}
\begin{proof}
Straightforward.
\end{proof}

\subsection{The four-valued ideal paraconsistent logic ${\sf J}_4$} \label{sectJ4}

As mentioned in Remark \ref{axiomHqi}, we know from Theorem 4.3 in \cite{blok:pig:01} that it is possible to obtain a standard (that is, without ``global'' inference rules) Hilbert calculus for a logic $\mathsf{L}^i_n$ for $i< n$ from the usual one for $\L_{n+1}$ by using translations. However, the calculi obtained in this manner can lack an intuitive meaning since they are defined in terms of the implication connective $\to$ of $\L_{n+1}$, that is naturally associated to the filter $F_1 =\{1\}$ but not
to the filter $F_{i/n} = \{i/n, \ldots, 1\}$, which is the one at work in ${\sf L}^i_n$.  Actually, the implication naturally associated to the filter $F_{i/n}$ is  $\Rightarrow^i_n$, which was considered above, for which modus ponens (MP) and the deduction-detachament theorem hold.

In this section we focus on the particular case of the (ideal paraconsistent) logic ${\sf J}_4 = {\sf L}^1_3$. ${\sf J}_4$ can be considered as a generalization to four values of the paraconsistent 3-valued logic ${\sf J}_3$ introduced by da Costa and D'Ottaviano in \cite{dot:dac:70} and briefly mentioned in Example~\ref{ExL3}.
For this logic a more natural signature $\Sigma$ will be
considered for describing it axiomatically in terms
of a deductive implication connective (in the sense of Definition~\ref{IdPar} item~(ii)) and a unary connective $*$  representing the square operation $x \otimes x$, which can be seen as a kind of `truth stresser' (see e.g. \cite{Ha01c}).
A soundness and completeness result  for this calculus proved by using a `separation' technique for truth-values will be presented. Note that dealing with logics ${\sf J}_q = \mathsf{L}^1_q$  for a prime $q>3$ appears to be much more complicated, and certainly it lies outside the scope of this paper.

The signature $\Sigma$ that will be used in the rest of the section is given by two unary connectives ${\ast}$ (square) and $\neg$ (negation), plus a binary connective  $\vee$ for disjunction. Abusing the notation, we formally define next ${\sf J}_4$ over this signature, and we will show later that it is an equivalent presentation of ${\sf L}^1_3$.

\begin{definition} \label{algA4}
${\sf J}_4$ is the matrix  logic $\langle {\bf A}_4, F_{1/3}\rangle$ over  $\Sigma$, where the algebra is ${\bf A}_4 = (\domLfin_4, \lor, \neg, *)$, with operations defined by the tables below:
$$
\begin{array}{|c||c|c|c|c|}  \hline
  \vee & 1 & 2/3 & 1/3  & 0\\ \hline \hline
  1 & 1 & 1 & 1 & 1 \\ \hline
 2/3 & 1 & 2/3 & 2/3 & 2/3 \\ \hline
  1/3 & 1 & 2/3 & 1/3  & 1/3\\ \hline
  0 & 1 & 2/3 & 1/3  & 0 \\ \hline
\end{array}
\hspace{1 cm}
\begin{array}{|c||c|c|} \hline
   & \neg & {\ast}  \\ \hline \hline
  1 & 0 & 1  \\ \hline
  2/3 & 1/3 & 1/3  \\ \hline
  1/3 & 2/3 & 0  \\ \hline
  0 & 1 & 0 \\ \hline
\end{array}
$$
\end{definition}

Observe that $\neg$ is \L ukasiewicz negation in $\algLfin_4$, while ${\ast}x=x \otimes x$ (with $\otimes$ being {\L}ukasiewicz strong conjunction) and $\vee$ is the lattice join in $\algLfin_4$.

In this signature $\Sigma$ the following derived connectives can be defined (as usual, the corresponding operators will be denoted using the same symbol): \\

\begin{tabular}{ll}
- & $\Delta(p) = {\ast}{\ast} p$ ; \\
- &  ${\sim}p = \Delta( \neg p)$ ; \\
- &  $p \Rightarrow r = {\sim}p \vee r$ ; \\
- &  $p \Leftrightarrow r = (p \Rightarrow r) \wedge (r \Rightarrow p)$ ;  \\
- &  $p \wedge r = \neg(\neg p \vee \neg r)$ ; \\
- &  $\nabla(p)= \neg{\sim}p$; \\
- &  $\alpha_{1/3}(p)=\nabla(p) \wedge {\sim}{\ast} p$; \\
- &  $\beta_{1/3}(p)=\alpha_{1/3}(p) \wedge {\ast}\neg p$.\\ \\
\end{tabular}
\ \\
It is easy to see that $\Delta$ is Monteiro-Baaz {\em Delta-operator}) and $\sim$ is G\"odel negation (${\sim}x=1$ if $x=0$, and $0$ otherwise). Note that $\sim$ actually coincides with ${\sim}^1_3$, and thus $\Rightarrow$ is nothing but $\Rightarrow^1_3$. Furthermore, $\nabla(x)=0$ if $x=0$, and $1$ otherwise; $\alpha_{1/3}(x)=1$ and $\beta_{1/3}(x)=1/3$ if $x=1/3$, and $0$ otherwise.

It is worth to remark that {\L}ukasiewicz implication is definable from these operators in the following way:  $$p \to r = ((\nabla(\neg p) \vee r) \wedge (\neg p \vee \nabla(r)) \wedge \neg \beta_{1/3}(r) )
 \vee (({\sim}p \wedge \alpha_{1/3}(r)) \vee (\alpha_{1/3}(p) \wedge \alpha_{1/3}(r))).$$
Then, the following result follows easily:

\begin{proposition} \label{algA4=algL4}
The algebras $\algLfin_4$ and ${\bf A}_4$ are functionally equivalent.
\end{proposition}

This means that the proposed operators over $\Sigma$ constitute an alternative presentation of the algebra $\algLfin_4$ underlying $\L_4$. Next we define  an axiomatic system for ${\sf J}_4$.

\begin{definition}  \label{calH4} The Hilbert  calculus  ${\sf H}_4$  for the logic ${\sf J}_4$, defined over the signature $\Sigma$, is given as follows:\\[2mm]
  {\em Axiom schemas:} those of {\sf CPL} over the signature $\{\vee, \Rightarrow, {\sim}\}$ plus
  \begin{itemize}
\item[(Ax1)]    $\neg{\sim}\alpha \Rightarrow \alpha$ \vspace{-0.2cm}
\item[(Ax2)]    $\alpha \vee \neg \alpha$ \vspace{-0.2cm}
\item[(Ax3)]   $\neg\neg\alpha \Leftrightarrow \alpha$ \vspace{-0.2cm}
\item[(Ax4)]    $\neg(\alpha \vee \beta) \Rightarrow \neg \alpha$ \vspace{-0.2cm}
\item[(Ax5)]    $\neg(\alpha \vee \beta) \Rightarrow \neg \beta$ \vspace{-0.2cm}
\item[(Ax6)]    $\neg \alpha \Rightarrow (\neg\beta \Rightarrow \neg(\alpha \vee \beta))$ \vspace{-0.2cm}
\item[(Ax7)]    ${\ast}\alpha \Rightarrow \alpha$ \vspace{-0.2cm}
\item[(Ax8)]    ${\ast}(\alpha \vee \neg \alpha)$ \vspace{-0.2cm}
\item[(Ax9)]    ${\ast}\alpha \Rightarrow {\sim}{\ast}\neg\alpha$ \vspace{-0.2cm}
\item[(Ax10)]    ${\ast}{\ast}\alpha \Leftrightarrow {\sim}\neg \alpha$ \vspace{-0.2cm}
\item[(Ax11)]    $\neg {\ast}\alpha \Leftrightarrow \neg\alpha$ \vspace{-0.2cm}
\item[(Ax12)]    ${\ast}(\alpha \vee \beta) \Leftrightarrow ({\ast}\alpha \vee {\ast}\beta)$
\end{itemize}
\noindent {\em Inference rule:}
  \begin{itemize}
\item[(MP)] $\displaystyle\frac{\alpha \ \ \ \ \alpha\Rightarrow
      \beta}{\beta}$
\end{itemize}
\end{definition}
Observe that, since (MP) is the only inference rule, ${\sf H}_4$ satisfies the deduction-detachment theorem w.r.t. the implication $\Rightarrow$: $\Gamma \cup\{ \alpha\} \vdash_{{\sf H}_4} \beta$ iff $\Gamma \vdash_{{\sf H}_4}  \alpha \Rightarrow \beta$, for every set of formulas $\Gamma \cup \{\alpha,\beta\}$. {On the other hand, it can be proved that $*(\alpha\Rightarrow \beta) \Rightarrow (*\alpha \Rightarrow *\beta)$ is derivable in ${\sf H}_4$, which gives additional support to consider $*$ as a truth stresser.}
Soundness of ${\sf H}_4$ can be proved straightforwardly.

\begin{proposition} [Soundness of ${\sf H}_4$] The calculus ${\sf H}_4$ is sound w.r.t. ${\sf J}_4$, that is: $\Gamma \vdash_{{\sf H}_4} \varphi$ implies that $\Gamma \vDash_{\sf J_4} \varphi$, for every finite set of formulas $\Gamma \cup \{\varphi\}$.
\end{proposition}

In order to prove completeness, since ${\sf H}_4$ is a finitary Tarskian logic, one can use the technique of maximal consistent sets of formulas. Indeed, for any set of formulas $\Gamma \cup \{\varphi \}$, if $\Gamma \nvdash_{{\sf H}_4} \varphi$ then, by Lindenbaum-{\L}os theorem, $\Gamma$ can be extended to a maximal set  $\Lambda$ such that $\Lambda  \nvdash_{{\sf H}_4} \varphi$. We will call the set $\Lambda$ {\em maximal \ntwrt\varphi}  in  ${\sf H}_4$.
Maximal sets w.r.t.\ a formula enjoy remarkable properties which directly follow from the axioms and rules of ${\sf H}_4$.

\begin{proposition}\label{MaxJ4}
Let $\Lambda$ be a maximal set \ntwrt\varphi\  in  ${\sf H}_4$. Then, $\Lambda$ is closed, i.e. for every formula $\psi$, $\Lambda \vdash\psi$ iff $\psi\in\Lambda$. Moreover, for any formulas $\alpha$ and $\beta$ the following conditions hold: \vspace{0.2cm}

\begin{tabular}{rl}
(1) &$\alpha \vee \beta \in \Lambda$ iff $\alpha \in \Lambda$ or $\beta \in \Lambda$;\\
(2) &$\alpha \not\in\Lambda$ iff ${\sim}\alpha \in\Lambda$;\\
(3) & $\alpha \Rightarrow \beta \in \Lambda$ iff $\alpha \not\in \Lambda$ or $\beta \in \Lambda$;\\
(4) &$\alpha \not\in\Lambda$ implies $\neg\alpha \in\Lambda$;\\
(5) & $\alpha \in\Lambda$ iff $\neg\neg\alpha \in\Lambda$;\\
(6) & $\neg{\sim}\alpha \in\Lambda$ implies $\alpha \in\Lambda$;\\
(7) & $\neg(\alpha \vee \beta) \in \Lambda$ iff $\neg\alpha \in \Lambda$ and $\neg\beta \in \Lambda$;\\
(8) & ${\ast}\alpha \in\Lambda$ implies $\alpha \in\Lambda$;\\
(9) & ${\ast}(\alpha \vee \beta) \in \Lambda$ iff ${\ast}\alpha \in \Lambda$ or ${\ast}\beta \in \Lambda$;\\
(10) & ${\ast}{\ast}\alpha\in\Lambda$ iff $\neg\alpha  \not\in\Lambda$;\\
(11) & $\neg{\ast}\alpha \in\Lambda$ iff $\neg\alpha \in\Lambda$;\\
(12) & ${\ast}\alpha \not\in\Lambda$ iff ${\ast}\neg\alpha \in\Lambda$.
\end{tabular}
\end{proposition}

Next we prove a Truth Lemma for ${\sf H}_4$.

\begin{lemma} [Truth Lemma for ${\sf H}_4$] \label{TL-J4} Let $\Lambda$ be a maximal set of formulas \ntwrt\varphi\  in  ${\sf H}_4$. Consider the following evaluation $e_\Lambda$ of propositional variables for ${\sf J}_4$:
$$(T) \hspace*{2cm}
e_\Lambda(\gamma) = \left\{\begin{array}{rl}
1           & \mbox{iff}\quad \gamma \in \Lambda, \mbox{and} \;\neg \gamma \not\in \Lambda\\[1mm]
2/3 & \mbox{iff}\quad  \gamma \in \Lambda,  \;\neg \gamma \in \Lambda, \mbox{and} \; {\ast} \gamma \in \Lambda\\[1mm]
1/3 & \mbox{iff}\quad  \gamma \in \Lambda,  \;\neg \gamma \in \Lambda, \mbox{and} \; {\ast} \gamma \not\in \Lambda\\[1mm]
0           & \mbox{iff}\quad  \gamma \not\in \Lambda.
\end{array}\right.
$$Then,  (T) holds for every complex formula $\gamma$.
\end{lemma}
\begin{proof}
The proof is done by induction on the complexity of the formula $\gamma$. If $\gamma$ is atomic then (T) holds by hypothesis.
Now, suppose (T) holds for every formula with complexity $\leq n$ (induction hypothesis -- IH) and let $\gamma$ be a formula with complexity $n$.
In order to prove (T) from (IH) by analyzing all the possible cases (namely, $\gamma = \neg\alpha$ or $\gamma = {\ast}\alpha$ or  $\gamma = \alpha \vee \beta$), each item of Proposition~\ref{MaxJ4} should be used.\footnote{Observe that it is enough to prove the `only if' part of (T), since the four conditions on the right-hand side are pairwise incompatible, and $e(\gamma)$ can only take one of the values $0, 1, 1/3, 2/3$. Thus, if, for instance, the first condition on the right-hand side of (T) holds  for a  given formula $\gamma$ then the other 3 conditions are false and so $e_\Lambda(\gamma) \not\in \{1/3,1,0\}$, by the `only if' part of (T). Hence,  $e_\Lambda(\gamma)$ must be $2/3$. This shows that the `if' part of (T) follows from the `only if' part.} The details are left to the reader.
\end{proof}

\begin{theorem} [Completeness of ${\sf H}_4$] \label{complJ4} The calculus ${\sf H}_4$ is complete w.r.t. ${\sf J}_4$, that is: $\Gamma \vDash_{{\sf J}_4} \varphi$ implies that $\Gamma \vdash_{{\sf H}_4} \varphi$, for every finite set of formulas $\Gamma \cup \{\varphi\}$.
\end{theorem}
\begin{proof}
Let $\Gamma\cup \{\varphi\}$ be a set of formulas in the language of ${\sf J}_4$
such that $\Gamma\nvdash_{{\sf H}_4} \varphi$.
By Lindenbaum-\L os, there exists a set $\Lambda$  maximal \ntwrt\varphi\ in ${\sf H}_4$
such that $\Gamma \subseteq \Lambda$.
Let $e_\Lambda$ be the evaluation defined as in the Truth Lemma~\ref{TL-J4}.
Then, it follows that $e_\Lambda(\gamma) \in F_{1/3}$ iff $\gamma \in \Lambda$, for every formula $\gamma$. Therefore $e_\Lambda$ is an evaluation such that $e_\Lambda[\Gamma] \subseteq F_{1/3}$ but $e_\Lambda(\varphi) = 0$ since $\varphi \not\in \Lambda$, hence $\Gamma \not\vDash_{{\sf J}_4} \varphi$.
\end{proof}

Recall that, from Theorem \ref{final-simple} and Remark \ref{rem-simple}, the Hilbert calculus $\overline{\sf H}_4$ obtained from  ${\sf H}_4$ by adding the explosion rule
$$(exp_1) \ \displaystyle \frac{\varphi \land \neg\varphi}{\bot} $$
(see Definition~\ref{Hbar}) is the axiomatization of the (only) proper extension of  ${\sf H}_4$ which is strongly maximal w.r.t. {\sf CPL}, and that is semantically characterized by the matrix logic
$$\bar{J}_4 =\langle {\bf A}_4 \times {\bf A}_{2}, F_{1/3}\times\{1\}  \rangle, $$
where ${\bf A}_2$ is the Boolean algebra over $\{0,1\}$  in the signature $\Sigma$, where the operator $\ast$ is defined  as ${\ast}x=x$.

\section{Conclusions} \label{concl}

In this paper we have been concerned with the study of maximality and strong maximality conditions among finite-valued {\L}ukasiewicz logics ${\sf L}^i_n$ with order filters as designated values. In particular, we have characterized the conditions under which a logic ${\sf L}^i_n$ is maximal w.r.t.\ {\sf CPL} and its unique extension  $\bar{{\sf L}}^i_n$ by an inference rule is strongly maximal w.r.t.\ classical logic. This allows us to show that, although they are not strongly maximal w.r.t.\ {\sf CPL},  the logics ${\sf L}^i_n$ with $n$ prime and $i/n \leq 1/2$ are in fact ideal paraconsistent logics.
{Thus, they provide interesting and well-motivated examples of  ideal  paraconsistent logical systems which are $(n+1)$-valued, in contrast with the $(n+2)$-valued  logics   $\mathcal{M}_{n+2}$ presented in~\cite{ArieliAZ11a} and reproduced here in Example~\ref{EjIdeal}, whose definition is somewhat {\em ad hoc}.}

As for future work, there are several interesting problems that we leave open in this paper.
{
Concerning maximality, a natural question is how to obtain a stronger version of Theorem~\ref{maxthm} which give us sufficient conditions to guarantee that a given matrix logic $L_1$ is {\em strongly maximal} w.r.t. another matrix logic $L_2$.}
On the other hand, notice that the study of strong maximality developed in Section~\ref{Joan} was heavily based on
results on the algebraic semantics associated to these systems by means of the Blok and Pigozzi's techniques. Thus, another interesting issue to be explored in future work is to obtain more examples of strong maximality for different families of algebraizable logics

Another question raised here is the axiomatization of the ideal paraconsistent logics ${\sf J}_{q+1}$ for $q > 3$ in a ``natural'' signature containing a deductive implication. As it was shown in Subsection~\ref{sectJ4}, the signature $\Sigma = \{ \lor, \neg, *\}$ is suitable for the case $q = 3$. Moreover, besides being apt for axiomatizing ${\sf J}_4 = {\sf L}^1_3$, it can be proved that the (non-paraconsistent) logic ${\sf L}^2_3$ can also be axiomatized over $\Sigma$ in a relatively simple way. Note that $\alpha \Rightarrow \beta = \neg\alpha \vee \beta$ defines a deductive implication in ${\sf L}^2_3$.

The fact that  {\L}ukasiewicz implication is definable in $\Sigma$ justifies the convenience of using that signature for dealing with the case $q=3$. However,  this property does not hold for any prime $q> 3$. Indeed, there are primes $q$ in which {\L}ukasiewicz implication of \L$_{q+1}$ cannot be defined over  $\Sigma$, e.g. $q = 17$. The study of the fragments of ${\sf L}^i_q$ in the signature $\Sigma$ is thus a different but closely related problem, which deserves future research.

\subsection*{Acknowledgements} The authors acknowledge partial support by the H2020 MSCA-RISE-2015 project SYSMICS. Coniglio was also financially supported by an individual research grant from CNPq, Brazil (308524/2014-4). Esteva and Godo also acknowledge partial support by the Spanish MINECO/FEDER project RASO (TIN2015- 71799-C2-1-P). Gispert also acknowledges partial support by the Spanish \linebreak MINECO/FEDER projects (MTM2016-74892 and MDM-2014-044) and grant 2017-SGR-95 of Generalitat de Catalunya.

\bibliographystyle{plain}

\end{document}